\documentclass[conference]{IEEEtran}
\usepackage{times}
\usepackage[numbers]{natbib}
\usepackage{multicol}
\usepackage[colorlinks,bookmarksopen,bookmarksnumbered,citecolor=red,urlcolor=red]{hyperref}
\usepackage{savesym}
\usepackage{amsmath,amssymb,amsfonts,amsthm}
\usepackage[utf8]{inputenc}
\usepackage{algorithm,algcompatible}

\algnewcommand\algorithmicto{\textbf{to}}
\algnewcommand\RETURN{\State \textbf{return} }
\usepackage[noend]{algpseudocode}
\usepackage{comment}
\usepackage{xspace}
\usepackage{scalerel}
\usepackage[caption=false,font=footnotesize]{subfig}
\usepackage{bm}
\usepackage{enumerate}
\usepackage{booktabs} % For prettier tables
\usepackage{graphicx}
\usepackage{mathtools}
\usepackage[dvipsnames]{xcolor}
\usepackage{dsfont}
\usepackage[group-separator={,}]{siunitx}
\usepackage[acronym]{glossaries}
\glsdisablehyper
\allowdisplaybreaks

\renewcommand{\baselinestretch}{0.95} % Trick to save space
\newif\ifmargincomments %A quick way of turning off margin comments for, say, arXiv submission
\margincommentstrue

\newif\ifdoubleblind
\doubleblindfalse

\ifmargincomments

\else

\fi

\begin{document}

%\title{Real-time and precise congestion-aware AMoD via Frank-Wolfe optimization}
%\title{Frank-Wolfe Optimization Algorithms for Autonomous Mobility-on-Demand \\ or \\ Real-time congestion-aware Autonomous Mobility-on-Demand flow \ or \\ Real-time congestion-aware autonomous mobility-on-demand via Frank-Wolfe optimization}
%\title{Congestion-aware Autonomous Mobility-on-Demand via Frank-Wolfe Optimization}
%\title{Scalable Routing for\\ Autonomous Mobility-on-Demand Systems\\ via Frank-Wolfe Optimization}
%\title{Scalable Routing for Autonomous Mobility-on-Demand via Frank-Wolfe Optimization}
\title{Scalable and Congestion-aware Routing for\\ Autonomous Mobility-on-Demand\\ via Frank-Wolfe Optimization}

\ifdoubleblind
% You will get a Paper-ID when submitting a pdf file to the conference system
\author{Author Names Omitted for Anonymous Review. Paper-ID 197}
\else
\author{
\authorblockN{Kiril Solovey, Mauro Salazar and Marco Pavone}
\\
\authorblockA{Autonomous Systems Lab, Stanford University, Stanford, CA 94305, USA \\ {\tt\small $\{$kirilsol,samauro,pavone$\}$@stanford.edu}}
}
\fi

% avoiding spaces at the end of the author lines is not a problem with
% conference papers because we don't use \thanks or \IEEEmembership

\maketitle

\begin{abstract}
We consider the problem of vehicle routing for Autonomous Mobility-on-Demand (AMoD) systems, wherein a fleet of  self-driving vehicles provides on-demand mobility in a given environment. Specifically, the task it to compute routes for the vehicles (both customer-carrying and empty travelling) so that travel demand is fulfilled and operational cost is minimized. The routing process must account for congestion effects affecting travel times, as modeled via a  volume-delay function (VDF). Route planning with VDF constraints is notoriously challenging, as such constraints compound the combinatorial complexity of the routing optimization process. Thus, current solutions for AMoD routing resort to relaxations of the congestion constraints, thereby trading optimality with computational efficiency.  In this paper, we present the first computationally-efficient approach for AMoD routing where VDF constraints are explicitly accounted for. We demonstrate that our approach is faster by at least one order of magnitude with respect to the state of the art, while providing higher quality solutions. From a methodological standpoint, the key technical insight is to   establish a mathematical reduction of the AMoD routing problem to the classical traffic assignment problem (a related vehicle-routing problem where empty traveling vehicles are not present). Such a reduction allows us to extend powerful algorithmic tools for  traffic  assignment, which combine the  classic Frank-Wolfe algorithm with modern techniques for pathfinding, to the AMoD routing problem. We provide strong theoretical guarantees for our approach in terms of near-optimality of the returned solution. 
\end{abstract}

\newcommand{\cupdot}{\mathbin{\mathaccent\cdot\cup}}

%tex tools
\newcommand{\ignore}[1]{}

\newcommand{\cpp}{C\raise.08ex\hbox{\tt ++}\xspace}

\def\P{\mathcal{P}} \def\C{\mathcal{C}} \def\H{\mathcal{H}}
\def\F{\mathcal{F}} \def\U{\mathcal{U}} \def\L{\mathcal{L}}
\def\O{\mathcal{O}} \def\I{\mathcal{I}} \def\S{\mathcal{S}}
\def\G{\mathcal{G}} \def\Q{\mathcal{Q}} \def\I{\mathcal{I}}
\def\T{\mathcal{T}} \def\L{\mathcal{L}} \def\N{\mathcal{N}}
\def\V{\mathcal{V}} \def\B{\mathcal{B}} \def\D{\mathcal{D}}
\def\W{\mathcal{W}} \def\R{\mathcal{R}} \def\M{\mathcal{M}}
\def\X{\mathcal{X}} \def\A{\mathcal{A}} \def\Y{\mathcal{Y}}
\def\L{\mathcal{L}}

\def\dS{\mathbb{S}} \def\dT{\mathbb{T}} \def\dC{\mathbb{C}}
\def\dG{\mathbb{G}} \def\dD{\mathbb{D}} \def\dV{\mathbb{V}}
\def\dH{\mathbb{H}} \def\dN{\mathbb{N}} \def\dE{\mathbb{E}}
\def\dR{\mathbb{R}} \def\dM{\mathbb{M}} \def\dm{\mathbb{m}}
\def\dB{\mathbb{B}} \def\dI{\mathbb{I}} \def\dM{\mathbb{M}}
\def\dZ{\mathbb{Z}}

\def\E{\mathbf{E}} % used for denoting expectation

\def\eps{\varepsilon}

\def\limn{\lim_{n\rightarrow \infty}}

\def\indicator{\mathds{1}}
\def\eqq{\coloneqq}

\def\Reals{\mathbb{R}}
\def\Naturals{\mathbb{N}}
\renewcommand{\leq}{\leqslant}
\renewcommand{\geq}{\geqslant}
\newcommand{\compl}{\mathrm{Compl}}

\newcommand{\np}{{\sc np}\xspace}
\newcommand{\argmin}{\operatornamewithlimits{argmin}}

\newtheorem{lemma}{Lemma}
\newtheorem{theorem}{Theorem}
\newtheorem{corollary}{Corollary}
\newtheorem{claim}{Claim}
\newtheorem{proposition}{Proposition}
\newtheorem{assumption}{Assumption}

\theoremstyle{definition}
\newtheorem{definition}{Definition}
\newtheorem{remark}{Remark}
\theoremstyle{plain}
\newtheorem{observation}{Observation}

\def\Im{\textup{Im}}

\def\aas{a.a.s.\xspace}
\def\0{\bm{0}}

\makeatletter
\def\thmhead@plain#1#2#3{%
  \thmname{#1}\thmnumber{\@ifnotempty{#1}{ }\@upn{#2}}%
  \thmnote{ {\the\thm@notefont#3}}}
\let\thmhead\thmhead@plain
\makeatother

\def\todo#1{\textcolor{blue}{\textbf{TODO:} #1}}
\def\new#1{\textcolor{magenta}{#1}}
\def\kiril#1{\textcolor{ForestGreen}{\textbf{Kiril:} #1}}
\def\old#1{\textcolor{red}{#1}}
\def\removed#1{\textcolor{green}{#1}}

\def\dt{\,\mathrm{d}t}
\def\dx{\,\mathrm{d}x}
\def\dy{\,\mathrm{d}y}
\def\ds{\,\mathrm{d}s}
\def\drho{\,\mathrm{d}\rho}

\def\amod{\textsc{AMoD}\xspace}
\def\amodl{\textsc{AMoD+L}\xspace}
\def\tap{\textsc{TAP}\xspace}
\def\fw{\textsc{FrankWolfe}\xspace}

\def\bpr{\textsc{bpr}\xspace}

\def\frankwolfe{\textsc{FrankWolfe}\xspace}
\def\allornothing{\textsc{AllOrNothing}\xspace}
\def\allornothingloss{\textsc{AllOrNothing-loss}\xspace}
\def\shortestpath{\textsc{ShortestPath}\xspace}

\def\x{\bm{x}}
\def\y{\bm{y}}
\def\z{\bm{z}}
\def\k{\bm{\kappa}}

\def\barE{\bar{E}}

%%% Local Variables:
%%% mode: plain-tex
%%% TeX-master: "main"
%%% End:

\newcommand{\blueb}[1]{\textcolor{blue}{#1}}
\newcommand{\redb}[1]{\textcolor{red}{#1}}

\definecolor{lightblue}{rgb}{0.60784,0.76078,0.90196}
\definecolor{darkblue}{rgb}{0.26667,0.44706,0.76863}
\definecolor{lightgreen}{rgb}{0.66275,0.81569,0.55686}
\definecolor{darkgreen}{rgb}{0.43922,0.67843,0.27843}
\definecolor{orange}{rgb}{0.92941,0.49020,0.19216}
\definecolor{yellow}{rgb}{1.00000,0.75294,0.00000}
\definecolor{grey}{rgb}{0.64706,0.64706,0.64706}
\definecolor{purple}{rgb}{0.51373,0.23529,0.04706}
% a
\newacronym{abk:amod}{\textsc{AMoD}\xspace}{Autonomous Mobility-on-Demand}
\newacronym{abk:iamod}{I-AMoD}{intermodal \gls{abk:amod}}
% b
\newacronym{abk:bpr}{BPR}{Bureau of Public Roads}
% c
% d
% e
% f
\newacronym{abk:ffcs}{FFCS}{free floating car sharing systems}
% g
\newacronym{abk:ghg}{GHG}{greenhouse gas}
% h
% i
\newacronym{abk:ilp}{ILP}{Integer Linear Programming}
% j
% k
% l
% m
\newacronym{abk:mcfp}{MCFP}{multi-commodity flow problem}
\newacronym{abk:mod}{MoD}{Mobility-on-Demand}
\newacronym{abk:mpc}{MPC}{Model Predictive Control}
% n
% o
% p
% q
% r
\newacronym{abk:rsmod}{RsMoD}{Ride-sharing Mobility-on-Demand}
\newacronym{abk:ramod}{RAMoD}{Ride-sharing Autonomous Mobility-on-Demand}
% s
\newacronym{abk:spp}{SPP}{shortest path problem}
\newacronym{abk:kdspp}{k-dSPP}{k-disjoint \gls{abk:spp}}
% t
\newacronym{abk:tap}{TAP}{Traffci Assignment Problem}
% u
% v
% w
% x
% y
% z	
%%%%%%%%%%%%%%%%%%%%%%%%%%%%%%%%%%%%%%%%%%%%%%%%%%%%%%%%%%%%%%%%%%%%%%%%%%%%%%%%%%%%%%%%%%%%%%%%%%%%%%%%%%%%%%%%%%%%%%%%%
% Commands
% \newcommand{}{}
%-Large-Arabic-------------------------------------------------------------------------------
% A
%\newcommand{}{A}
\newcommand{\eqMatrix}{\mathbf{A}}
% B
%\newcommand{}{B}
% C
%\newcommand{}{C}
% D
%\newcommand{}{D}
\newcommand{\ineqMatrix}{\mathbf{D}}
% E
%\newcommand{}{E}
% F
%\newcommand{}{F}
% G
%\newcommand{}{G}
% H
%\newcommand{}{H}
% I
%\newcommand{}{I}
% J
%\newcommand{}{J}
% K
%\newcommand{}{K}
% L
%\newcommand{}{L}
% M
%\newcommand{}{M}
% N
%\newcommand{}{N}
% O
%\newcommand{}{O}
% P
%\newcommand{}{P}
% Q
%\newcommand{}{Q}
% R
%\newcommand{}{R}
% S
%\newcommand{}{S}
\newcommand{\source}{S_{o}}
\newcommand{\sink}{S_{d}}
% T
%\newcommand{}{T}
\newcommand{\resCost}{T^{\mathrm{c}}}
\newcommand{\resPred}{T^{\mathrm{p}}}
% U
%\newcommand{}{U}
% V
%\newcommand{}{V}
\newcommand{\vTime}{V_\mathrm{T}}
\newcommand{\vDistR}{V_\mathrm{D,R}}
\newcommand{\vDistS}{V_\mathrm{D,P}}
\newcommand{\vEnergy}{V_\mathrm{E}}
\newcommand{\vCO}{V_{\mathrm{CO}_2}}
\newcommand{\vQuadratic}{V_\mathrm{Q}}
% W
%\newcommand{}{W}
% X
%\newcommand{}{X}
% Y
%\newcommand{}{Y}
% Z
%\newcommand{}{Z}
%-Small-Arabic-------------------------------------------------------------------------------
% a
%\newcommand{}{a}
% b
%\newcommand{}{b}
\newcommand{\eqRHS}{\mathbf{b}}
% c
%\newcommand{}{c}
\newcommand{\capacity}{c}
\newcommand{\capacityRoad}{c^{\mathrm{R}}}
\newcommand{\capacityPublic}{c^{\mathrm{P}}}
\newcommand{\cost}{J}
\newcommand{\costM}{J_\mathrm{M}}
\newcommand{\costE}{J_{\mathrm{CO}_2}}
\newcommand{\transpCoeff}{\mathbf{c}^\top}
% d
%\newcommand{}{d}
\newcommand{\destination}{d}
\newcommand{\ineqRHS}{\mathbf{d}}
% e
%\newcommand{}{e}
\newcommand{\energyR}{e_{\mathrm{R},ij}}
\newcommand{\sC}{s_{\mathrm{CO}_2}}
\newcommand{\energyS}{e_{\mathrm{S},ij}}
\newcommand{\energySperpaxmile}{\bar{e}_{\mathrm{P}}}
% f
%\newcommand{}{f}
\newcommand{\flow}[2]{f_m\left(#1,#2\right)}
\newcommand{\flowzero}[2]{f_m\left(#1,#2\right)}
\newcommand{\flowReba}[2]{f_0\left(#1,#2\right)}
% g
%\newcommand{}{g}
% h
%\newcommand{}{h}
% i
%\newcommand{}{i}
% j
%\newcommand{}{j}
% k
%\newcommand{}{k}
% l
%\newcommand{}{l}
% m
%\newcommand{}{m}
% n
%\newcommand{}{n}
% o
%\newcommand{}{o}
\newcommand{\origin}{o}
% p
%\newcommand{}{p}
\newcommand{\route}{p}
\newcommand{\pSubway}{p_\mathrm{P}(i,j)}
\newcommand{\pRoad}{p_\mathrm{R}(i,j)}
\newcommand{\pOrigin}[1]{p_\mathrm{O}\left(#1\right)}
\newcommand{\pDest}[1]{p_\mathrm{D}\left(#1\right)}
% q
%\newcommand{}{q}
% r
%\newcommand{}{r}
\newcommand{\request}{r}
\newcommand{\weight}{r_\mathrm{E}}
% s
%\newcommand{}{s}
% t
%\newcommand{}{t}
\newcommand{\traveltime}{t}
\newcommand{\tRoad}{\tau_\mathrm{R}(i,j)}
% u
%\newcommand{}{u}
% v
%\newcommand{}{v}
% w
%\newcommand{}{w}
% x
%\newcommand{}{x}
\newcommand{\decVec}{\mathbf{x}}
% y
%\newcommand{}{y}
% z
%\newcommand{}{z}
%-Large-Calligraphic-------------------------------------------------------------------------
% A
%\newcommand{}{\mathcal{A}}
\newcommand{\setOfArcs}{\mathcal{A}}
\newcommand{\setOfArcsRoad}{\mathcal{A}_{\mathrm{R}}}
\newcommand{\setOfArcsSubway}{\mathcal{A}_{\mathrm{P}}}
\newcommand{\setOfArcsPedestrian}{\mathcal{A}_{\mathrm{W}}}
\newcommand{\setOfArcsCommute}{\mathcal{A}_{\mathrm{C}}}
%\newcommand{\setOfArcsSP}{\mathcal{A}_{\mathrm{SP}}}
% B
%\newcommand{}{\mathcal{B}}
% C
%\newcommand{}{\mathcal{C}}
% D
%\newcommand{}{\mathcal{D}}
% E
%\newcommand{}{\mathcal{E}}
%\newcommand{\setOfArcs}{\mathcal{E}}
%\newcommand{\setOfArcsRoad}{\mathcal{E}_{\mathrm{R}}}
%\newcommand{\setOfArcsPublic}{\mathcal{E}_{\mathrm{P}}}
%\newcommand{\setOfArcsCommute}{\mathcal{E}_{\mathrm{C }}}
% F
%\newcommand{}{\mathcal{F}}
% G
%\newcommand{}{\mathcal{G}}
\newcommand{\GraphRoad}{\mathcal{G}_\mathrm{R}}
\newcommand{\GraphSubway}{\mathcal{G}_\mathrm{P}}
\newcommand{\GraphPedestrian}{\mathcal{G}_\mathrm{W}}
% H
%\newcommand{}{\mathcal{H}}
% I
%\newcommand{}{\mathcal{I}}
% J
%\newcommand{}{\mathcal{J}}
% K
%\newcommand{}{\mathcal{K}}
\newcommand{\setOfModes}{\mathcal{K}}
% L
%\newcommand{}{\mathcal{L}}
% M
%\newcommand{}{\mathcal{M}}
\newcommand{\setOfRequestsNumber}{\mathcal{M}}
% N
%\newcommand{}{\mathcal{N}}
% O
%\newcommand{}{\mathcal{O}}
% P
%\newcommand{}{\mathcal{P}}
\newcommand{\pathSet}{\mathcal{P}}
% Q
%\newcommand{}{\mathcal{Q}}
% R
%\newcommand{}{\mathcal{R}}
\newcommand{\setOfRequests}{\mathcal{R}}
% S
%\newcommand{}{\mathcal{S}}
\newcommand{\subSet}{\mathcal{S}}
% T
%\newcommand{}{\mathcal{T}}
% U
%\newcommand{}{\mathcal{U}}
% V
%\newcommand{}{\mathcal{V}}
\newcommand{\setOfVertices}{\mathcal{V}}
\newcommand{\setOfVerticesRoad}{\mathcal{V}_{\mathrm{R}}}
\newcommand{\setOfVerticesSubway}{\mathcal{V}_{\mathrm{P}}}
\newcommand{\setOfVerticesPedestrian}{\mathcal{V}_{\mathrm{W}}}
% W
%\newcommand{}{\mathcal{W}}
% X
%\newcommand{}{\mathcal{X}}
% Y
%\newcommand{}{\mathcal{Y}}
% Z
%\newcommand{}{\mathcal{Z}}
%-Large-Scr-------------------------------------------------------------------------
% A
%\newcommand{}{\mathscr{A}}
\newcommand{\setOfMultiArcs}{\mathscr{A}}
% B
%\newcommand{}{\mathscr{B}}
% C
%\newcommand{}{\mathscr{C}}
% D
%\newcommand{}{\mathscr{D}}
% E
%\newcommand{}{\mathscr{E}}
% F
%\newcommand{}{\mathscr{F}}
% G
%\newcommand{}{\mathscr{G}}
\newcommand{\multigraph}{\mathscr{G}}
\newcommand{\multigraphC}{\mathscr{G}'}
% H
%\newcommand{}{\mathscr{H}}
% I
%\newcommand{}{\mathscr{I}}
% J
%\newcommand{}{\mathscr{J}}
% K
%\newcommand{}{\mathscr{K}}
% L
%\newcommand{}{\mathscr{L}}
% M
%\newcommand{}{\mathscr{M}}
% N
%\newcommand{}{\mathscr{N}}
% O
%\newcommand{}{\mathscr{O}}
% P
%\newcommand{}{\mathscr{P}}
\newcommand{\setOfArcsIndex}{\mathscr{P}}
% Q
%\newcommand{}{\mathscr{Q}}
% R
%\newcommand{}{\mathscr{R}}
% S
%\newcommand{}{\mathscr{S}}
% T
%\newcommand{}{\mathscr{T}}
% U
%\newcommand{}{\mathscr{U}}
% V
%\newcommand{}{\mathscr{V}}
% W
%\newcommand{}{\mathscr{W}}
% X
\newcommand{\decvecSet}{\mathcal{X}}
% Y
%\newcommand{}{\mathscr{Y}}
% Z
%\newcommand{}{\mathscr{Z}}
\newcommand{\setOfMultiarcIndex}{\mathscr{Z}}
%-Large-Greek--------------------------------------------------------------------------------
% Alpha
%\newcommand{}{A}
% Beta
%\newcommand{}{B}
% Gamma
%\newcommand{}{\Gamma}\\
\newcommand{\neighborOne}{\Gamma}
% Delta
%\newcommand{}{\Delta}
% Epsilon
%\newcommand{}{E}
% Zeta
%\newcommand{}{Z}
% Eta
%\newcommand{}{H}
% Theta
%\newcommand{}{\Theta}
% Iota
%\newcommand{}{I}
% Kappa
%\newcommand{}{K}
% Lambda
%\newcommand{}{\Lambda}
% Mu
%\newcommand{}{M}
% Nu
%\newcommand{}{N}
% Xi
%\newcommand{}{\Xi}
% Omicron
%\newcommand{}{O}
% Pi
%\newcommand{}{\Pi}
% Rho
%\newcommand{}{P}
% Sigma
%\newcommand{}{\Sigma}
% Tau
%\newcommand{}{T}
% Upsilon
%\newcommand{}{\Upsilon}
% Phi
%\newcommand{}{\Phi}
% Chi
%\newcommand{}{X}
% Psi
%\newcommand{}{\Psi}
% Omega
%\newcommand{}{\Omega}
%-Small-Greek--------------------------------------------------------------------------------
% alpha
%\newcommand{}{\alpha}
\newcommand{\requestrate}{\alpha}
% beta
%\newcommand{}{\beta}
% gamma
%\newcommand{}{\gamma}
% delta
%\newcommand{}{\delta}
\newcommand{\cutsetSucc}[1]{\delta^{\text{+}}\left(#1\right)}
\newcommand{\cutsetPred}[1]{\delta^{\text{-}}\left(#1\right)}
% epsilon
%\newcommand{}{\varepsilon}
% zeta
%\newcommand{}{\zeta}
% eta
%\newcommand{}{\eta}
% theta
%\newcommand{}{\vartheta}
% iota
%\newcommand{}{\iota}
\newcommand{\iteration}{\iota}
% kappa
%\newcommand{}{\kappa}
% lambda
%\newcommand{}{\lambda}
\newcommand{\lgeq}{\mathbf{\lambda}}
\newcommand{\dualEq}{\lambda_{ij}}
\newcommand{\dualCust}{\lambda_\mathrm{C}}
\newcommand{\dualVeh}{\lambda_\mathrm{R}}
\newcommand{\dualCustTilde}{\tilde{\lambda}_\mathrm{C}}
\newcommand{\dualVehTilde}{\tilde{\lambda}_\mathrm{R}}
% mu
%\newcommand{}{\mu}
\newcommand{\dualCR}{\mu_\mathrm{cR}}
\newcommand{\lgineq}{\mathbf{\mu}}
\newcommand{\dualIneq}{\mu}
\newcommand{\dualcR}{\mu_\mathrm{cR}}
\newcommand{\dualcS}{\mu_\mathrm{cP}}
% nu
%\newcommand{}{\nu}
% xi
%\newcommand{}{\xi}
% omicron
%\newcommand{}{o}
% pi
%\newcommand{}{\pi}
% rho
%\newcommand{}{\rho}
% sigma
%\newcommand{}{\sigma}
% tau
%\newcommand{}{\tau}
% upsilon
%\newcommand{}{\upsilon}
% phi
%\newcommand{}{\phi}
% chi
%\newcommand{}{\chi}
% psi
%\newcommand{}{\psi}
% omega
%\newcommand{}{\omega}
%-Other-------------------------------------------------------------------
\newcommand{\arc}{(i,j)}
\newcommand{\multiarc}{(i,j,z)}
\newcommand{\bool}[1]{\mathds{1}_{#1}}

%\section{Introduction}
%Similar background as CARS.
\section{Introduction}
Mobility in urban environments is becoming a major issue on the global scale~\cite{LevyBuonocoreEtAl2010}.
The main reasons are an increasing population with higher mobility demands and a slowly adapting infrastructure~\cite{CIA:2018}, resulting in serious congestion problems.
In addition, the usage of public transit is dropping, whilst mobility-on-demand operators such as Uber and Lyft are increasing their operation on urban roads, increasing further congestion~\cite{Berger2018,Molla2018,Siddiqui2018}.
For instance, the yearly cost of congestion in the US has doubled between 2007 and 2013~\cite{Schrank2007,TuttleCowles2014}, and in Manhattan cars are traveling about 15\% slower compared to five years ago~\cite{Hu2017}.

Space limitations and a largely fixed infrastructure make congestion an issue difficult to address in urban environments.
%, whereby ``brute-force'' solutions such as increasing road capacity are not applicable.
While existing public transportation systems need to be extended to ease congestion, it is important to adopt technological innovations improving the efficiency of urban transit.
The advent of cyber-physical technologies such as autonomous driving and wireless communications will enable the deployment of \gls{abk:amod} systems, i.e., fleets of self-driving cars providing on-demand mobility in a one-way vehicle-sharing fashion (see Fig.~\ref{fig:AMoD}). Specifically, such a system is designed to carry passengers from their origins to their destinations, potentially in an intermodal  fashion (i.e., utilizing several modes of transportation), and to assign empty vehicles to new requests.
The main advantage of \gls{abk:amod} systems is that they can be controlled by a \emph{central} operator simultaneously computing routes for customer-carrying vehicles and \emph{rebalancing} routes for empty vehicles, thus enabling a \emph{system-optimal} operation of this transportation system. This way, \gls{abk:amod} systems could replace current taxi and ride-hailing services and reduce the global cost of travel~\cite{SpieserTreleavenEtAl2014}.

%In order to ease congestion, \gls{abk:amod} systems must be operated in a congestion-aware manner, considering the {\em endogenous} impact of vehicles' flows on travel time.
Conversely to conventional navigation providers computing the fastest route by passively considering congestion in an exogenous manner, \gls{abk:amod} systems controlled by a central operator enable one to consider the \emph{endogenous} impact of the single vehicles' routes on road traffic and travel time, and can thus be operated in a congestion-aware fashion.

%\mpmargin{Conventional}{this paragrpah should be streamlined, as it contains some repetitions} navigation providers compute the fastest route by passively considering congestion in an exogenous fashion, whilst neglecting the endogenous impact of the single vehicle on traffic, and thus steering the transit network towards the user equilibrium~\cite{Wardrop1952}. Conversely, large fleets of self-driving vehicles controlled by a central operator enable one to consider the endogenous impact of the single vehicles' routes on road traffic and thus to actively influence congestion. This way, \gls{abk:amod} systems can be operated in a congestion-aware fashion.
%In this paper, we devise a computationally-efficient route-planning algorithm for \gls{abk:amod} systems
%and combine it with a \emph{volume-delay function} describing the impact of traffic on travel time
%to rapidly compute system-optimal routes. To distinguish our contribution from the status quo, we now review related literature on this topic.
\begin{figure}[t]
	\centering\includegraphics[width=\columnwidth]{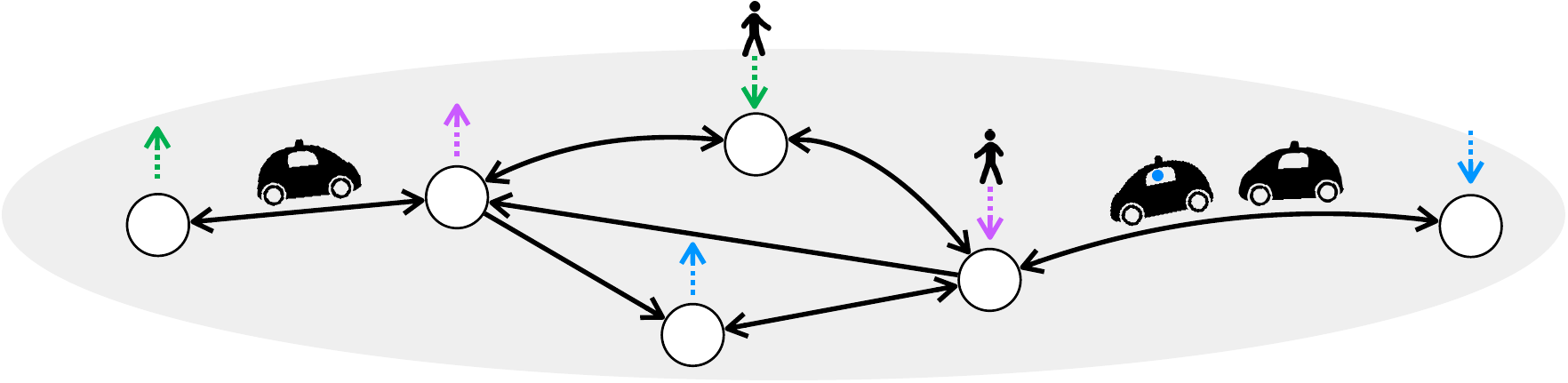}\vspace{-5pt}
	\caption{The AMoD network. The white circles represent intersections and the black arrows denote road links. The  dotted arrows represent pick-up and drop-off locations for single customers.}
	\label{fig:AMoD}\vspace{-15pt}
\end{figure}

{\em Statement of contributions:}
We introduce a computationally-efficient approach for congestion-aware \amod routing to minimize the system cost---the total cost of executing the routing scheme over all the vehicles in the system. To the best of our knowledge, this is the first method that  takes into consideration the full representation of the volume-delay function that estimates the travel time based on the amount of traffic.
%In contrast, previous approaches for \amod have either disregarded congestion affects, or used a piecewise-linear approximation of the volume delay function \todo{citations needed}.
Moreover, we demonstrate that our approach is faster by at least one order of magnitude than previous work (see Section~\ref{sec:related}) for congestion-aware \amod, while being more accurate in terms of congestion estimation.

On the algorithmic side, we develop a reduction which transforms the \amod routing problem into a Traffic Assignment Problem (\tap), where the latter does not involve rebalancing of empty vehicles. We then prove mathematically that an optimal solution for the latter \tap instance yields a solution to our original \amod problem with the following properties: (i) The majority (e.g., 99\% in our experiments) of rebalancing demands are fulfilled and (ii) the system cost of the solution is upper bounded by the system optimum where 100\% of the rebalancing demands are fulfilled. (We note that, in practice, the unfulfilled rebalancing demand, being just a small fraction -- say, $<1\%$ -- can be addressed via post-processing heuristic strategies with minimal impact on cost.)

Such a reformulation of the \amod problem allows us to leverage state-of-the-art techniques for \tap, that can efficiently compute a congestion-aware system optimum. In particular, we employ the classic Frank-Wolfe algorithm~\cite{FrankWolfe56, Patriksson15}, which is paired with modern shortest-path techniques, such as contraction hierarchies~\cite{GeisbergerETAL12} (both of which are implemented in recent open-source libraries~\cite{BuchETAL18,DibbeltETAL16}). This allows us to compute in a few seconds (on a commodity laptop) \amod routing schemes for a realistic test case over Manhattan, New York, consisting of $\num{156000}$ passenger travel requests.

\emph{Organization}: The remainder of this paper is structured as follows.
In Section~\ref{sec:related} we provide a review of related work.
In Section~\ref{sec:preliminaries} we formally define the instances of \tap and \amod we are concerned with in this work. There we also discuss the assumptions of our model and possible limitations. In Section~\ref{sec:convex} we provide a description of the Frank-Wolfe algorithm for \tap. Our main theoretical contribution is given in Section~\ref{sec:alg_amod}, where we describe our approach for \amod by casting it into \tap, and develop its mathematical properties. In Section~\ref{sec:experiments} we demonstrate the power of our approach and test its scalability on realistic inputs. We conclude the paper with a discussion and future work in Section~\ref{sec:future}.

\section{Related Work}\label{sec:related}
%These contributions provide the \textit{analysis} of congestion-aware routing schemes, but lack control algorithms for both passenger-carrying and empty vehicles.

%\kiril{I feel that there is too much information and emphasis on BPR.}

There exist several approaches to study \gls{abk:amod} systems, spanning from simulation models~\cite{HoerlRuchEtAl2018, LevinKockelmanEtAl2017,MaciejewskiBischoffEtAl2017} and queuing-theoretical models~\cite{ZhangPavone2016,IglesiasRossiEtAl2016} to network-flow models~\cite{PavoneSmithEtAl2012,RossiZhangEtAl2017,SpieserTreleavenEtAl2014}.
On the algorithmic side, the \textit{control} of \gls{abk:amod} systems has been mostly based on network flow models employed in a receding-horizon fashion~\cite{IglesiasRossiEtAl2018,TsaoIglesiasEtAl2018,TsaoMilojevicEtAl2019}, and thresholded approximations of congestion effects~\cite{RossiZhangEtAl2017}, also accounting for the interaction with public transit~\cite{SalazarRossiEtAl2018} and the power-grid~\cite{RossiIglesiasEtAl2018}. In such a framework, cars can travel through a road at free-flow speed until a fixed capacity of the road is reached. At that point, no more cars can traverse it.
Such models result in optimization problems solvable with off-the-shelf linear-programming solvers---making them very well suitable for \emph{control} purposes---but lacking accuracy when accounting for congestion phenomena, which are usually described with \emph{volume-delay functions} providing a relationship between traffic flow and travel time. In particular, the \gls{abk:bpr} developed the most commonly used volume-delay function~\cite{BPR1964}, which has been applied to problems ranging from dynamic estimation of congestion~\cite{RivasInmaculadaEtAl2016} to route planning in agent-based models~\cite{AslamLimEtAl2012, ManleyChengEtAl2014}.
Against this backdrop, a piecewise-affine approximation of the \gls{abk:bpr} function is presented in~\cite{SalazarTsaoEtAl2019} and combined with convex relaxation techniques to devise a congestion-aware routing scheme for \gls{abk:amod} systems resulting in a quadratic program. Nevertheless, in large urban environments with several thousand transportation requests such approaches usually lead to computational times of the order of minutes, possibly rendering them less suitable for real-time control purposes.
 
Mathematically, \amod can be viewed as an extension of \tap, where the latter ignores the cost and impact of rebalancing empty vehicles. Historically, \tap was introduced to model and quantify the impact of independent route choices of human drivers on themselves, and the system as a whole (see ~\cite{Patriksson15, Sheffi85}). Algorithmic approaches for \tap typically assume a \emph{static} setting in which travel patterns do not change with time, allowing to cast the problem into a multi-commodity minimum-cost flow~\cite{Ahuja93}, which can then be formulated as a \emph{convex programming} problem. One of the most popular tools for convex programming in the context of \tap is the Frank-Wolfe algorithm~\cite{FrankWolfe56}. What makes this algorithm particularly suitable for solving \tap is that its direction-finding step corresponds to multiple shortest-path
%\footnote{The shortest-path problem corresponds to finding the shortest path between two given vertices of a given graph. It is important to note that in the context of the Frank-Wolfe algorithm, the multitude of paths should be computed independently of each other, disregarding the impact each path has on the others.}
queries (we expand on this point in Section~\ref{sec:convex}). Recent advances in pathfinding tailored for transportation networks~\cite{BastETAL2016}, including contraction hierarchies~\cite{DibbeltETAL16, GeisbergerETAL12}, have made the Frank-Wolfe approach remarkably powerful. In particular, a recent work~\cite{BuchETAL18} introduced a number of improvements for pathfinding, and combines those with the Frank-Wolfe method. Notably, the authors present experimental results for \tap, where their approach computes within seconds a routing scheme for up to~3 million requests over a network of a large metropolitan area. Nevertheless, such an approach is not directly applicable to \amod problems as it would not account for the rebalancing of empty vechicles.

Finally, we mention that \amod is closely related to \emph{Multi-Robot Motion Planning} (MRMP), which consists of computing collision-free paths for a group of physical robots, moving them from the initial set of positions to their destination. MRMP has been studied for the setting of discrete~\cite{SharonETAL15,Yu18,YuLaValle16} and continuous~\cite{CapETAL15,ShomeETAL2019,SoloveyETAL15} domains, respectively. The unlabeled variant of MRMP~\cite{AdlerETAL15,SolHal16,TurpinETAL14b,TurpinETAL14a,YuLaValle12}, which involves also target assignment, is reminiscent of the rebalancing empty vehicles in \amod, as such vehicles do not have a priori assigned destinations.
\section{Preliminaries}\label{sec:preliminaries}
In this section we provide a formal definition of  \tap and \amod, as our work will exploit the tight mathematical relation between these two problems.

The basic ingredient in our study is the road network, which is modeled as a directed graph $G=(V,E)$: Each vertex $v\in V$  represents either a physical road intersection or points of interest on the road. Each edge $(i,j)\in E$ represents a road segment connecting two vertices $i,j\in V$.
To model the travel times along the network, every edge $(i,j)\in E$ is associated with a cost function $c_{ij}:\dR_+\rightarrow \dR_+$, which represents the travel time along the edge as a function of flow (i.e., traffic) along the edge, denoted by $x_{ij}\geq 0$. In order to accurately capture the cost, every edge $(i,j)$ has two additional attributes: the capacity of the edge $\kappa_{ij}>0$, which can be viewed as the amount of flow beyond which the travel time will increase rapidly, and the free-flow travel time $\phi_{ij}>0$, i.e., the \emph{nominal} travel time when $x_{ij}= 0$. We mention that those attributes are standard when modeling traffic (see, e.g.,~\cite{Patriksson15}).
The time-invariant nature of this model captures the \emph{average} value of the flows for a certain time period.

To compute $c_{ij}$ we use the \bpr function~\cite{BPR1964}, which is often employed in practice. It is noteworthy that our approach presented below can be modified to work with other functions such as the modified Davidson cost~\cite{Akcelik78}. Specifically, we define the cost function as
\begin{align*}
    c_{ij}(x_{ij})=\bpr(x_{ij},\kappa_{ij},\phi_{ij}) \eqq \phi_{ij}\cdot \left(1+\alpha\cdot \left(\frac{x_{ij}}{\kappa_{ij}}\right)^{\beta}\right),
\end{align*}
where typically $\alpha=0.15$ and $\beta=4$.

Travel demand is represented by  passenger requests  $OD=\{(\lambda_m,o_m,d_m)\}_{m=1}^M$, %\mpmargin{where}{this explanation can be improved, is not very elegant} \mbox{$\lambda_m>0$} represents the amount of customers per unit time , and $o_m,d_m\in V$ are the origin and destination, respectively.
where \mbox{$\lambda_m>0$} represents the amount of customers willing to travel from the origin node $o_m\in V$ to the destination node $d_m\in V$ per time unit.
%That is, a given request $m\in M$ indicates that there are on average $\lambda_m$ customers wishing to travel from $o_m$ to $d_m$ within the considered time period.

\subsection{Traffic Assignment}
Here we provide a mathematical formulation of traffic assignment. We denote by $x_{ijm}\in \dR_+$ the flow induced by request $m\in M$ on edge $(i,j)\in E$. We introduce the following constraint which ensures that the amount of flow associated with each request is maintained when a flow enters and leaves a given vertex. The amount of flow corresponding to the request $m\in M$, leaving $o_m$ and entering $d_m$ must match the demand flow $\lambda_m$ as
\begin{align}
    \sum_{j\in V_i^+}x_{ijm}-\sum_{j\in V_i^-}x_{jim}&=\lambda_{im},\quad \forall i\in V, m\in M,  \label{eq:tap:flow}
\end{align}
\[\textup{where }\lambda_{im}\eqq\begin{cases} \lambda_m,\quad & \text{if }o_m=i,\\ -\lambda_m,\quad & \text{if }d_m=i,\\0,\quad & \text{otherwise},\end{cases}\]
and $V_i^-\eqq\left\{j\middle|(i,j)\in E\right\},V_i^+\coloneqq \left\{j\middle|(j,i)\in E\right\}$, denote heads and tails of edges leaving and entering $i\in V$, respectively. We also impose non-negative flows as
\begin{align}
    x_{ijm}& \geq 0, \quad\forall (i,j)\in E. \label{eq:tap:positive}
\end{align}

The objective of \tap is specified in the following definition. Informally, the goal is to minimize the total travel time experienced by the users in the system, that is the sum of travel times for each individual request $m$.
\begin{definition}\label{def:tap}
    The \emph{traffic-assignment problem} (\tap) consists of minimizing the expression 
\begin{equation}
    F_E(\x)=\sum_{(i,j)\in E}x_{ij} c_{ij}(x_{ij}),\quad \textup{subject to~(\ref{eq:tap:flow}), (\ref{eq:tap:positive})}, 
\end{equation}
where $\x\eqq \left\{x_{ij}=\sum_{m\in M}x_{ijm}|(i,j)\in E\right\}$.
\end{definition}

\subsection{Autonomous Mobility-on-Demand}
In an \gls{abk:amod} system, the formulation of \tap captures only partially the cost of operating the full system. In particular, vehicles need to perform two types of tasks: (i) \emph{occupied} vehicles drive passengers from their origins to their destinations; (ii) after dropping passengers off at their destination,  \emph{empty} vehicles need to drive to the next origin nodes, where passengers will be picked up. Indeed, the formulation of \tap above only captures the cost associated with (i), but not (ii). Another crucial difference between \tap and \amod, which makes the latter significantly more challenging, is the fact that the travel destinations of empty vehicles are not given a priori and should be computed by the algorithm.
Thus, we extend the model to include also rebalancing empty vehicles and define $x_{ijr}$ as the rebalancing flow of empty vehicles over $(i,j)\in E$.
%We force empty vehicles to be rebalanced from destinations of \msmargin{previous}{there is no notion of time here. delete} requests to origins of \msmargin{subsequent}{delete} requests as
We force empty vehicles to be rebalanced from destination nodes to origin nodes as
\begin{align}
    \sum_{j\in V_i^+}x_{ijr}-\sum_{j\in V_i^-}x_{jir}&=r_i,\quad \forall i\in V,\label{eq:amod:flow}
\end{align}
\[\textup{for }r_i\eqq \sum_{m\in M}\left(\indicator\{d_m=i\}-\indicator\{o_m=i\}\right)\lambda_m, \] where  $\indicator\{\cdot\}$ is a boolean indicator function. Observe that nodes with more arriving than departing passengers do not require rebalancing. We use $R\eqq \sum_{i\in V}\indicator\{r_i>0\}r_i$ to denote the total number of rebalancing requests and enforce non-negativity of rebalancing flows as
\begin{align}
    x_{ijr} &\geq 0,\quad \forall(i,j)\in E.\label{eq:amod:positive}
\end{align}
    
\begin{definition}
The \emph{autonomous-mobility-on-demand problem} (\amod) consists of minimizing the expression $F_E(\hat{\x})$, subject to  (\ref{eq:tap:flow}), (\ref{eq:tap:positive}), (\ref{eq:amod:flow}), (\ref{eq:amod:positive}), where 
$\hat{\bm{x}}\eqq \left\{\hat{x}_{ij}\eqq x_{ij}+x_{ijr}\right\}_{(i,j)\in E}$. 
\end{definition}

\subsection{Discussion}
A few comments are in order.
First, we make the assumption that mobility requests do not change in time. This assumption is justified in cities where transportation requests change slowly with respect to the average travel time~\cite{Neuburger1971}.
Second, the model describes vehicle routes as fractional flows and it does not account for the stochastic nature of the trip requests and exogenous traffic. Given the mesoscopic perspective of our study, such an approximation is in order.
Moreover, given the computational effectiveness of the approach, our algorithm is readily implementable in real-time in a receding horizon fashion, whereby randomized sampling algorithms can be adopted to compute integer-valued solutions with near-optimality guarantees~\cite{Rossi2018}.
Third, we assume exogenous traffic to follow habitual routes and neglect the impact of our decisions on the traffic base load, leaving the inclusion of reactive flow patterns to future work.
Fourth, we model the impact of road traffic on travel time with the \gls{abk:bpr} function~\cite{BPR1964}, which is well established and, despite it does not account for microscopic traffic phenomena such as traffic lights, serves the purpose of route-planning on the mesoscopic level.
Finally, we constrain the capacity of the vehicles to one single customer, which is in line with current trends, and leave the extension to ride-sharing to future research~\cite{CapAlonso18,TsaoMilojevicEtAl2019}.
\section{Convex Optimization for \tap}\label{sec:convex}
In this section we describe the Frank-Wolfe method for convex optimization, which will later be used for solving \amod. First, we have the following statement concerning the convexity of \tap.

\begin{claim}[(Convexity)]
    \tap (Definition~\ref{def:tap}) is a convex problem. \label{claim:convex}
\end{claim}
\begin{proof}
    Given a specific edge $(i,j)\in E$, observe that the derivative of the expression $x_{ij}c_{ij}(x_{ij})$ is strictly increasing, which implies that it is convex. As the expression $F_E(\x)$ consists of a sum of convex functions, it is convex as well. 
\end{proof}

\subsection{The Frank-Wolfe Algorithm}
Due to the convexity of the problems introduced, we can leverage convex optimization to solve \tap, and consequently \amod, as we will see later on. Specifically, we use the Frank-Wolfe algorithm (\fw), which is well suited to our setting and has achieved impressive practical results for large-scale instances of \tap in a recent work~\cite{BuchETAL18}. 

Before introducing \fw, it should be noted that  
it is typically employed to minimize the \emph{user-equilibrium} cost function captured by 
\begin{align}
    \bar{F}_E(\x)=\sum_{(i,j)\in E}\int_0^{x_{ij}}\bar{c}_{ij}(s)\ds,\label{eq:ue}
\end{align} 
for some $\bar{c}_{ij}:\dR_+\rightarrow \dR_+$, whereas we are interested in computing the \emph{system optimum} corresponding to the minimum of $F_E(\x)=\sum_{(i,j)\in E}x_{ij}c_{ij}(x_{ij})$, using $c_{ij}$ as defined in the previous section. However, we can enforce the user-equilibrium reached by selfish agents to correspond to the system optimum by using the \emph{marginal costs} \[\bar{c}_{ij}(x_{ij})=\frac{\mathrm{d}}{\mathrm{d}x_{ij}}\left(x_{ij}c_{ij}(x_{ij})\right)=c_{ij}(x_{ij})+x_{ij}c'_{ij}(x_{ij}),\]
which quantifies the sensitivity of the total cost with respect of small changes in flows.
Specifically, to compute the system optimum, we only need to apply \fw to minimize $\bar{F}_E(\x)$ (Equation~\ref{eq:ue}). (See more information on this transformation in~\cite{Patriksson15}.) 
The algorithm below will be presented with respect to $\bar{F}_E$.

The following pseudo-code (Algorithm~\ref{alg:fw}) presents a simplified version of \fw, which is based on~\cite[Chapter 4.1]{Patriksson15}. The algorithm begins with an initial solution $\x^0$, which satisfies~(\ref{eq:tap:flow}), (\ref{eq:tap:positive}). To obtain $\x_0$, one can, for instance, assign each request $(\lambda_m,o_m,d_m)$ to the shortest route over the traffic-free graph $G$, while ignoring the flows of the other users.  

\begin{algorithm}
\caption{\frankwolfe$(\bar{F}_E,G,OD)$}
	\begin{algorithmic}[1]
		\State{$\x^0\gets \text{feasible solution for \tap}$; $k\gets 0$}
		\While {stopping criterion not reached}
		\State{$\y^k\gets \argmin_{\y} \bar{F}_E(\x^k)+\nabla \bar{F}_E(\x^k)^T(\y-\x^k)$,  s.t. $\y^k$ satisfies (\ref{eq:tap:flow}), (\ref{eq:tap:positive})}
		\State{$\alpha_k\gets \argmin_{\alpha\in [0,1]}\bar{F}_E(\x^k+\alpha(\y^k-\x^k))$}
		\State{$\x^{k+1}\gets \x^k+\alpha_k(\y^k-\x^k)$; $k\gets k+1$}
		\EndWhile	
		\State 	\Return $\x^{k}$ 
	\end{algorithmic} \label{alg:fw}
\end{algorithm}

In each iteration $k$ of the algorithm, the following steps are performed. In line~3 a value of $\y^k$ minimizing the expression $\bar{F}_E(\x^k)+\nabla \bar{F}_E(\x^k)^T(\y^k-\x^k)$, which satisfies~(\ref{eq:tap:flow}), (\ref{eq:tap:positive}), is obtained. It should be noted that this corresponds to solving a linear program with respect to $\y^k$, due to the fact that $\x^k$ is already known, and one is working with the gradient of $\bar{F}_E$ rather than the function itself. We will say a few more words about this computation below. In line~4 a scalar $\alpha_k\in [0,1]$ is found, such that $\bar{F}_E\left(\x^k+\alpha_k(\y^k-\x^k)\right)$ is minimized, which corresponds to solving a single-variable optimization problem, which can be done efficiently. At the end of the iteration in line~5 the solution is updated to be a linear interpolation between $\x^k$ and $\y^k-\x^k$. The last value of $\x^k$ computed before the stopping criteria has been reached, is returned in the end. Due to the convexity of the problem, it is guaranteed that as \mbox{$k\rightarrow \infty$}, $\x^k$ converges to the optimal solution of \tap.

\subsection{All-or-nothing Assignment}
What makes \fw particularly suitable for solving \tap is the special structure of the task of computing $\y^k$ which minimizes $\bar{F}_E(\x^k)+\nabla \bar{F}_E(\x^k)^T(\y^k-\x^k)$. First, observe that it is equivalent to minimizing the expression $\nabla\bar{F}_E(\x^k)^T\y^k$. Next, notice that for any $(i,j)\in E$ it holds that $\frac{\partial}{\partial x_{ij}}\bar{F}_E(\x^k)=\bar{c}_{ij}(x^k_{ij})$, where $x^k_{ij}$ is the value corresponding to $(i,j)$ of $\x^k$. That is, every variable $y^k_{ij}$ is multiplied by $\bar{c}_{ij}(x^k_{ij})$. Thus, minimizing the expression $\nabla\bar{F}_E(\x^k)^T\y^k$ while satisfying~\eqref{eq:tap:flow}, \eqref{eq:tap:positive} is equivalent to independently assigning the shortest route for every request $(\lambda_m,o_m,d_m)$, over the graph $G$, where the cost of traversing the edge $(i,j)$ is independent of the traffic passing through it, and is equal to $(\nabla\bar{F}_E(\x^k))_{ij}$. 

This operation is known as All-or-Nothing assignment, due to the fact that each request is assigned to one specific route. Its pseudo code is given below (Algorithm~\ref{alg:aon}). The \shortestpath routine returns a vector $\y^k_m$, where for every $(i,j)\in E$ that is found on the shortest path from $o_m$ to $d_m$ on $G$, weighted by $\nabla\bar{F}_E(\x^k)$, $\y^k_{m,ij}=1$, and $\y^k_{m,ij}=0$ otherwise.  

\begin{algorithm}
\caption{\allornothing$(G,\nabla\bar{F}_E(\x^k),OD)$}
	\begin{algorithmic}[1]
		\For {$m\in M$}
		\State{$\y^k_m\gets \shortestpath(G,\nabla\bar{F}_E(\x^k),o_m,d_m)$}
		\EndFor	
		\State 	\Return $\y^{k}\eqq \sum_{m\in M}\lambda_m y^k_m$
	\end{algorithmic}\label{alg:aon}
\end{algorithm}
\section{AMoD as TAP}\label{sec:alg_amod}
In this section we establish an equivalence between \tap and \amod. In particular, we show that a given \amod problem can be transformed into a \tap, such that a solution to the latter, which is obtained by \frankwolfe, yields a solution to the former. 

The crucial difference between the two problems is that in \tap every vehicle has a specific origin and destination vertex, whereas in \amod this is not the case. In particular, while in \amod empty rebalancers originate in specific destination vertices of user requests, the destinations of these rebalancers can in theory be any of the origin vertices. However, we show that this gap can be bridged by supplementing the original graph $G$ with an additional ``dummy'' vertex, and connecting to it edges emanating from all the vertices that need to be rebalanced. We then set the costs of the edges to guarantee an almost complete rebalancing, i.e., only a small fraction of the rebalancing requests will not be fulfilled. In the remainder of this section we provide a detailed description of the approach and proceed to analyze its theoretical guarantees. 

\subsection{The construction}
We formally describe the structure of this new graph $G'=(V',E')$, where $V'=V\cup \{n\}$, and 
\[E'=E\cup \left\{(i,n)\middle| i\in V\textup{ and }r_i<0\right\}.\]
Recall that $r_i<0$ indicates that there are fewer user requests arriving to $i\in V$ than there are departing from the vertex, which implies that rebalancers should be sent to this vertex. The vertex $n\not\in V$ is new and will serve as dummy target vertex for all the rebalancers. See example in Figure~\ref{fig:construction}.

To ensure that a sufficient number of rebalancers will arrive at each vertex that needs to be rebalanced, we assign to every edge $(i,n)$ 
the cost $c_{in}(x_{in})=\bpr(x_{in},\kappa_{in},\phi_{in})$, where $\kappa_{in}=-r_i$, and $\phi_{in}=L$, 
where $L$ is a large constant whose value will be determined later on.

The final ingredient in transforming \amod to \tap is providing excess vehicles with  specific origins and destinations. Given the original set of requests $OD$, for every $i\in V$ such that $r_i>0$ we add the request $(r_i,o_i,n)$, where $r_i$ is its intensity. This yields the  extended requests set~$OD'$.

\begin{figure}%[h!]
%	\vspace{-.1in}
    \centering
    \includegraphics[width=0.5\columnwidth]{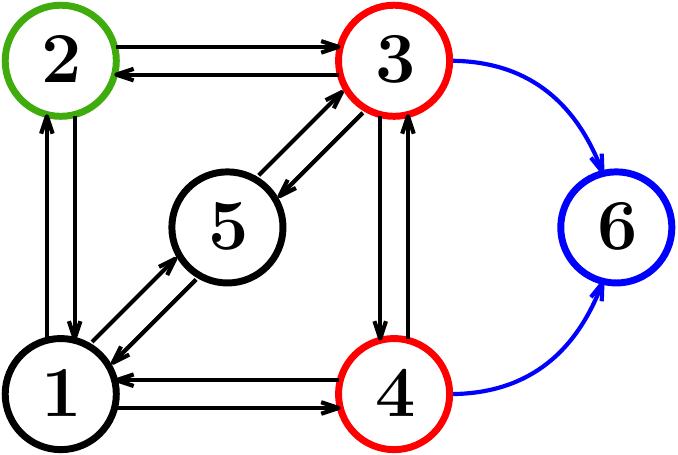}\vspace{-5pt}
    \caption{A simple example for the construction. The graph~$G$ consists of the vertices $1$ to $5$ and the (black) edges between them. The demand for~$G$ is $OD=\{(2,1,2),(1,2,4),(1,3,4),(2,4,1),(2,4,2)\}$, where each triplet denotes the intensity, origin, and destination, respectively. Red vertices indicate shortage of incoming vehicles (e.g., four passengers depart from vertex $4$, but only two vehicles arrive), green indicates excess (e.g, four vehicles terminate in vertex $2$ but only two passengers leave), black represents vertices with  met demand (e.g., vertex $1$ where two vehicles terminating, and two passengers departing). 
    Consequently, the graph~$G'$ additionally contains the dummy vertex $6$, and edges $(3,6),(4,6)$ drawn in blue, originating from vertices of $G$ with shortage. Accordingly, the capacity is set to $\kappa_{3,6}=1,\kappa_{4,6}=2$. $OD'$ extends $OD$ with the request $(3,2,6)$. Observe that its intensity corresponds to the total excess in vertices $3,4$.}
	\label{fig:construction}
\vspace{-15pt}
\end{figure}

As each free or occupied vehicle has a specific destination, we can think of the \amod problem as a new \tap problem over the graph $G'$ and the extended set of requests $OD'$. As rebalancers are no longer needed to be considered separately from the users, we may redefine $x_{ij}$ to be the total flow along an edge $(i,j)\in E'$, including users and rebalancers. Denote $\bar{E}\eqq E'\setminus E$. The cost function for the corresponding \tap is $F_{E'}(\x)=F_{E}(\x)+F_{\bar{E}}(\x)$, where $F_{\bar{E}}(\x)=\sum_{(i,n)\in \bar{E}}x_{in}c_{in}(x_{in})$.

We show in the remainder of this section that after choosing~$L$ correctly, then $\x^*$, which minimizes the system optimum $F_{E'}(\x^*)$ under the constraints (\ref{eq:tap:flow}), (\ref{eq:tap:positive}), with respect to $G',OD'$, represents a high-quality solution to the \amod problem, in which the majority of rebalancing requests are fulfilled. 
It is worth clarifying that the cost of the obtained flow is represented by $F_{E}(\x)$.

\subsection{Analysis}\label{sec:theory}
First, we note that the new objective function $F_{E'}$ remains convex owing to the fact that for every new dummy edge its cost function is monotone and increasing with respect to its flow (see Claim~\ref{claim:convex}).
Given a vector assignment $\x$ for \tap over $G'$, it will be useful to split it into variables $\x_E$ corresponding to the edges $E$, and variables $\x_{\bar{E}}$ corresponding to $\bar{E}$.

The motivation for setting the specific capacity value $\kappa_{in}$ to edge $(i,n)\in \bar{E}$  is given in the following lemma. Recall that $R\eqq \sum_{i\in V}\indicator\{r_i>0\}r_i$. 
\begin{lemma}[(Optimal assignment for dummy edges)]\label{lem:opt}
    Let $\x^*$ minimize $F_{\barE}(\x^*)$, under the constraint that $\sum_{(i,n)\in \bar{E}}x^*_{in}=R$. Then $\x^*_{\barE}=\k$, where $\k=\left\{\kappa_{in}\middle| (i,n)\in \bar{E}\right\}$.
\end{lemma}
\begin{proof}
    Note that $F_{\barE}(\x_{\bar{E}})$ is convex, and feasible for the constraint $\sum_{x_{in}\in\x_{\bar{E}}} x_{in}=R$. Thus, it has a unique minimum. To find it, we shall use Lagrange multipliers. 
    
    Let $g(\x)\eqq \sum_{j=1}^m x_j$and define the Lagrangian \mbox{$\L(\x,\lambda)\eqq F_{\bar{E}}(\x)-\lambda(g(\x)-R)$}.
    For any $x_j\in \x$,
    \begin{align*}
        \frac{\partial}{\partial x_j} \L = L\left(1+5\left(\frac{x_j}{\kappa_j}\right)^4\right)-\lambda, \text{ and } \frac{\partial}{\partial \lambda} \L =R-g(\x).
    \end{align*}The constraint $\frac{\partial}{\partial x_j} \L=0$ yields that \mbox{$x_j = \kappa_j\left(\frac{\lambda}{5L}-\frac{1}{5}\right)^{1/4}$},
    which is then plugged into $\frac{\partial}{\partial \lambda}\L=0$, where we use the fact that $\sum_{j=1}^\ell \kappa_j =R$, to yield 
    $\lambda = 6L$. We then substitute $\lambda$ to yield $x_j=\kappa_j$, which concludes the proof. 
\end{proof}

We arrive at the main theoretical contribution of the paper, which states that $L$ can be tuned to obtain a solution where the fraction of unfulfilled rebalancing requests $\delta>0$ is as small as desired. Notice that for a given solution $\x$, the expression  $\frac{\|\x_{\barE}-\k\|_1}{2R}$ represents the fraction of unfulfilled requests. 
\begin{theorem}[(Bounded fraction of unfulfilled requests)] \label{thm:amod}
    Let $\x^*\eqq\argmin_{\x} F_{E'}(\x)$ subject to~(\ref{eq:tap:flow}), (\ref{eq:tap:positive}).
    For every $\delta\in (0,1]$ exists $L_\delta\in (0,\infty)$ such that if $L>L_\delta$ then  $\frac{\|\x^*_{\barE}-\k\|_1}{2R}\leq \delta$.
\end{theorem}
 
\begin{proof}
Let $\x^0$ be an assignment satisfying~(\ref{eq:tap:flow}), (\ref{eq:tap:positive}) such that  $\x^0_{\barE}=\k$, which minimizes the expression $F_{E'}(\x^0)$. That is, $\x^0$ minimizes $F_{E'}$ over all $\x$ which fully satisfies the rebalancing constraints. (Observe that such $\x^0$ represents the optimal solution of the original \amod problem.)
If $\x^0$ turns out to yield the minimum of $F_{E'}(\x)$ without conditioning on $\x^0_{\barE}=\k$ then the result follows. Thus, we assume otherwise. 
Fix $\delta\in (0,1]$ and let $\x^\delta$ represent any assignment, satisfying~(\ref{eq:tap:flow}), (\ref{eq:tap:positive}), and for which it holds that $\frac{\|\x^\delta_{\barE}-\k\|_1}{2R}>\delta$.

Our goal is to find a value of $L_\delta$ such that for any $L>L_\delta$ it holds that 
$F_{E'}(\x^\delta)>F_{E'}(\x^0)$. This implies that using such $L$ we are guaranteed that if a solution returned by \frankwolfe will have at most $\delta$ unfulfilled request. This is equivalent to requiring that \[F_{\barE}(\x^\delta)-F_{\barE}(\x^0)>F_E(\x^0)-F_E(\x^\delta).\]
Notice that by Lemma~\ref{lem:opt}, it holds that  $F_{\barE}(\x^\delta)-F_{\barE}(\x^0)>0$.

Recovering a precise upper bound for the expression $F_E(\x^0)-F_E(\x^\delta)$, which depends on $\delta$, is quite difficult. We therefore resort to a crude (over-)estimation of it, which is the value $F_E(\x^0)$. Thus, we wish to find $L$ such that $F_{\barE}(\x^\delta)-F_{\barE}(\x^0)>F_E(\x^0)$.

Define $\Delta_{in}:=x_{in}-\kappa_{in}$, for every $(i,n)\in \barE$, where \mbox{$x_{in}\in \x^\delta_{\barE}$}. For every such $x_{in}$ the following holds:
{\small
\begin{align*}\allowdisplaybreaks
    c_{in}(& x_{in}) = L\left(1+0.15 \left(\frac{x_{in}}{\kappa_{in}}\right)^{4}\right) \\ & \geq L \left(1+0.15 \left(1+\frac{4\Delta_{in}}{\kappa_{in}}\right)\right) = L  \left(1.15+\frac{0.6\Delta_{in}}{\kappa_{in}}\right),
\end{align*}
}where the inequality follows from Bernoulli's inequality. Also, note that $c_{in}(\kappa_{in})=L\cdot 1.15$. Thus,
    {\small
    \begin{align*}\allowdisplaybreaks
        F&_{\barE}(\x^\delta)=\sum_{(i,n)}x_{in} c_{in}(x_{in}) \geq \sum_{(i,n)}(\kappa_{in}+\Delta_{in}) L \left(1.15+\frac{0.6\Delta_{in}}{\kappa_{in}}\right) \\ & = \sum_{(i,n)}L\left(1.15\cdot \kappa_{in} + 1.75\cdot \Delta_{in}+\frac{0.6\Delta_{in}^2}{\kappa_{in}}\right) \\ & = \sum_{(i,n)}1.15 L \kappa_{in} +\sum_{(i,n)}1.75L\Delta_{in}+ \sum_{(i,n)}L\frac{0.6\Delta_{in}^2}{\kappa_{in}}\\ & = F_{\barE}(\x^0)  +0+ \sum_{(i,n)}L\frac{0.6\Delta_{in}^2}{\kappa_{in}} \geq  F_{\barE}(\x^0)+\frac{0.6L}{R}\sum_{(i,n)}\Delta^2_{in}\\
        &\geq F_{\barE}(\x^0)+\frac{0.6L}{R} \cdot \ell^{-1}\left(\sum_{(i,n)}|\Delta_{in}|\right)^2 \\ & = F_{\barE}(\x^0)+\frac{0.6L}{R\ell}\left\|\x^\delta_{\bar{E}}-\x^0_{\bar{E}}\right\|_1^2  \geq F_{\barE}(\x^0) +2.4RL\ell^{-1} \delta^2,
    \end{align*}}where the second to last inequality is due to Cauchy-Schwarz, and $\ell$ is the number of dummy edges. 
    
    Thus, we have just shown that for any $\x^\delta$ it holds that $F_{\barE}(\x^\delta)-F_{\barE}(\x^0)\geq 2.4RL\ell^{-1}\delta^2$. To conclude, for \mbox{$L=\frac{F_E(\x^0)}{2.4R\ell^{-1}\delta^2}$}, it follows that $F_{E'}(\z^\delta)>F_{E'}(\z^0)$, which implies that any value $\x^*$ satisfying the constraints~(\ref{eq:tap:flow}), (\ref{eq:tap:positive}) and minimizing $F_{E'}(\x)$, also guarantees that $\frac{\|\x^*_{\barE}-\k\|_1}{2R}\leq \delta$.
\end{proof}

We will consider the practical aspects of computing a proper $L$ in Section~\ref{sec:experiments}. The next corollary is the final piece of the puzzle. It proves that when using a proper $L$, not only that $\delta$ is bounded, but also the value $F_E(\x^*)$ is at most $F_E(\x^0)$. 

\begin{corollary}[(Bounded cost of routing)]\label{thm:amod}
    Fix $\delta\in (0,1]$ and let $L\in (0,\infty)$ such that $L>L_{\delta}$.
    Then (i) $\frac{\|\x^*_{\barE}-\k\|_1}{2R}\leq \delta$, and (ii) $F_{E}(\x^*)< F_E(\x^0)$,  
    where $\x^*=\argmin_{\x} F_{E'}(\x)$ under constraints~(\ref{eq:tap:flow}), (\ref{eq:tap:positive}).
\end{corollary}
\begin{proof}
    Let $\x^*,\x^0$ be as defined in the previous proof. It is clear that $x^*$ satisfies (i). Now, observe that by definition $F_{E'}(\x^*)< F_{E'}(\x^0)$, and by Lemma~\ref{lem:opt}, $F_{\bar{E}}(\x^0)<F_{\bar{E}}(\x^*)$. Then the following derivation proves (ii):
    \begin{equation*}
        F_{E}(\x^*)+F_{\bar{E}}(\x^*) < F_{E}(\x^0)+F_{\bar{E}}(\x^0) < F_{E}(\x^0)+F_{\bar{E}}(\x^*). \qedhere
    \end{equation*}
\end{proof}
\section{Experimental results}\label{sec:experiments}
In this section, we use experimental results to demonstrate the power of our approach for \amod routing via a reduction to \tap (Section~\ref{sec:alg_amod}) on a real-world case study. In the first set of experiments in Section~\ref{sec:results} we validate experimentally the theory developed in Section~\ref{sec:theory}. In summary, we observe that the approach yields near-optimal solutions within a few seconds, where most (more than 99\%) of the rebalancing requests are fulfilled, when $L$ is properly tuned. We then test the scalability of the approach on scenarios involving as much as 600k user requests, where we observe that running times and convergence rates scale only linearly with the size of input. In the final set of experiments we demonstrate the benefit of the approach over previous methods. 

\subsection{Implementation Details}
All results were obtained using a commodity laptop equipped with 2.80GHz $\times$ 4 core i7-7600U CPU, and 16GB of RAM, running 64bit Ubuntu 18.04 OS. The \cpp implementation of the \frankwolfe algorithm is adapted (with merely minor changes) from the  \texttt{routing-framework}, which was developed for~\cite{BuchETAL18}. For shortest-path computation in the \allornothing routine, we use \emph{contraction hierarchies}~\cite{GeisbergerETAL12}, which are embedded in the \texttt{routing-framework}, and are in turn based upon the code in the \texttt{RoutingKit} (see~\cite{DibbeltETAL16}). Running times reported below are for a 4-core parallelization. 

In our experiment we observed that in some situations, the first few iterations of \frankwolfe route most of the rebalancers to a select set of dummy edges, so that the actual flow is far larger than the capacity. This leads to numeric overflows when working with the standard \cpp \texttt{double} and to failure of the program. We mention that this phenomenon was not observed for the modified Davidson cost function~\cite{Akcelik78}, linearized at $95\%$ of the capacity, which has a more gentle gradient. To alleviate the problem when working with \bpr, one can resort to using \texttt{long double} or linearize the value of the function after a certain threshold is reached. We chose the latter, by linearizing at $500\%$ of the capacity. 
%(for this value, \bpr multiplies the free-flow travel time by $94$). 
We emphasize that this does not affect the final outcome of the algorithm.

Finally, we mention that we experimented with a few cost functions (including linear, and modified Davidson) for the dummy edges, until we settled on \bpr, which yields the best convergence rates. 

\subsection{Data}
Similarly to~\cite{SalazarTsaoEtAl2019}, our experiments are conducted over Manhattan in New York City, where the OD-pairs are inferred from taxi rides that occurred during the morning peak hour on March 1st, 2012. As in~\cite{SalazarTsaoEtAl2019}, we assume to centrally control all  ride-hailing vehicles in Manhattan, and accordingly scale taxi rides requests by a factor of $6$.
The total number of real user OD-pairs in our experiments is thus $6\times \num{25960}$ (unless stated otherwise). 
In order to take into consideration the fact that autonomous vehicles need to share the road with private vehicles, which should increase the overall cost of travelling along edges, we introduce a parameter of exogenous traffic (as was done in~\cite{SalazarTsaoEtAl2019}). In particular, for a non-dummy edge $(i,j)\in E$, with a flow $x_{ij}$, we assign the cost $c_{ij}\left(x_{ij}+x^{e}_{ij}\right)$, where $x^{e}_{ij}$ denotes the exogenous flow. For simplicity, we set this value so that the fraction of $x^{e}_{ij}$ over the capacity $\kappa_{ij}$ of the edge $\kappa_{ij}$ is the same, over all edges. That is, we choose a value $\gamma_{\text{exo}}\geq 0$ and set $x^e_{ij}/\kappa_{ij}=\gamma_{\text{exo}}$.  Unless stated otherwise, $\gamma_{\text{exo}}=0.8$, which approximates the scenario of the rush-hour traffic, mentioned above. The underlying road-map $G=(V,E)$ was extracted from an Open Street Map~\cite{HakWeb08}, where $|V|=\num{1352},|E|=\num{3338}$.

\subsection{Results}\label{sec:results}
Before proceeding to the experiments we mention that in some results we compare the outcome of our algorithm with an optimal value, denoted by $\text{OPT}$. To obtain it, we run the algorithm, with a corresponding set of parameters, for $\num{10000}$ iterations. This provides a good approximation of the real optimum. E.g., for the final iteration of the algorithm, when $L=96$ minutes, the relative difference in the real and dummy cost is $\num{1.64}\cdot 10^{-7}$ and $\num{2.91}\cdot 10^{-7}$, respectively.  The terms real and dummy costs correspond to $F_E$ and $F_{\bar{E}}$, respectively. 
\vspace{5pt}

\noindent \textbf{Validation of the theory.} Our first set of experiments is designed to validate the convergence of the approach, and the theory presented in~Section~\ref{sec:alg_amod}. In summary, we observe that with already small $L$ a solution where a large majority of rebalancing requests are fulfilled is achieved. Moreover, the system cost for the rebalaced system is also very close to the optimal value. Importantly, this is achieved within only $\num{100}$ iterations of \frankwolfe, corresponding to around $15$ seconds.

In order to test how the value of $L$ affects the fraction of unbalanced requests $\delta$, as stated in Theorem~\ref{thm:amod}, and the  real and dummy costs of  solution, i.e., $C_r\eqq F_E(\x),C_d\eqq F_{\bar{E}}(\x)$, respectively, we experiment with values of $L$ in the range of $\num{3}$ to $\num{192}$ minutes (see Figure~\ref{fig:basic}).

In terms of the fraction of unfulfilled requests $\delta$, as Theorem~\ref{thm:amod} states, increasing $L$ reduces this value. For instance, when $L=3$, $\delta=0.109$, but already when $L=48$ then $\delta=0.009$, after 100 iterations. It should be noted that a small value of $\delta$ is reached only when $C_d$ is noticeably larger than $C_r$ (see middle and bottom plots in Figure~\ref{fig:basic}, for comparison).

This implies that our estimation for $L$ suggested in Theorem~\ref{thm:amod} is quite conservative. From a practical point of view, we recommend iterating over $L$ using binary search until a desired value of $\delta$ is achieved.

We now discuss the relation between the magnitude of $L$ and $C_r$. We observe that $C_r$ increases with $L$. This follows from the fact that a smaller $\delta$ requires rebalancers to increase the total length of their trips in order to accommodate more requests. Here we can also observe a possible drawback in setting $L$ to be needlessly large, as it takes more iterations to settle on the correct value of $C_r$. This follows from the fact that to obtain a smaller $\delta$ requires more iterations. 

Lastly, observe that already after approximately 50 iterations, all three values reach a plateau, i.e., increasing the number of iterations only slightly changes the corresponding value. Also note that values of OPT, computed for $L=96$, are very close to  corresponding values for the same $L$ after only $100$ iterations. 
This indicates that a small number of iterations suffices to reach a near-optimal value, be it $\delta$, $C_r$, or $C_d$. 

\begin{figure}%[h!]
%	\vspace{-.1in}
    \centering
    \includegraphics[width=0.9\columnwidth,trim={0.5cm 1.5cm 1.5cm 2.5cm},clip]{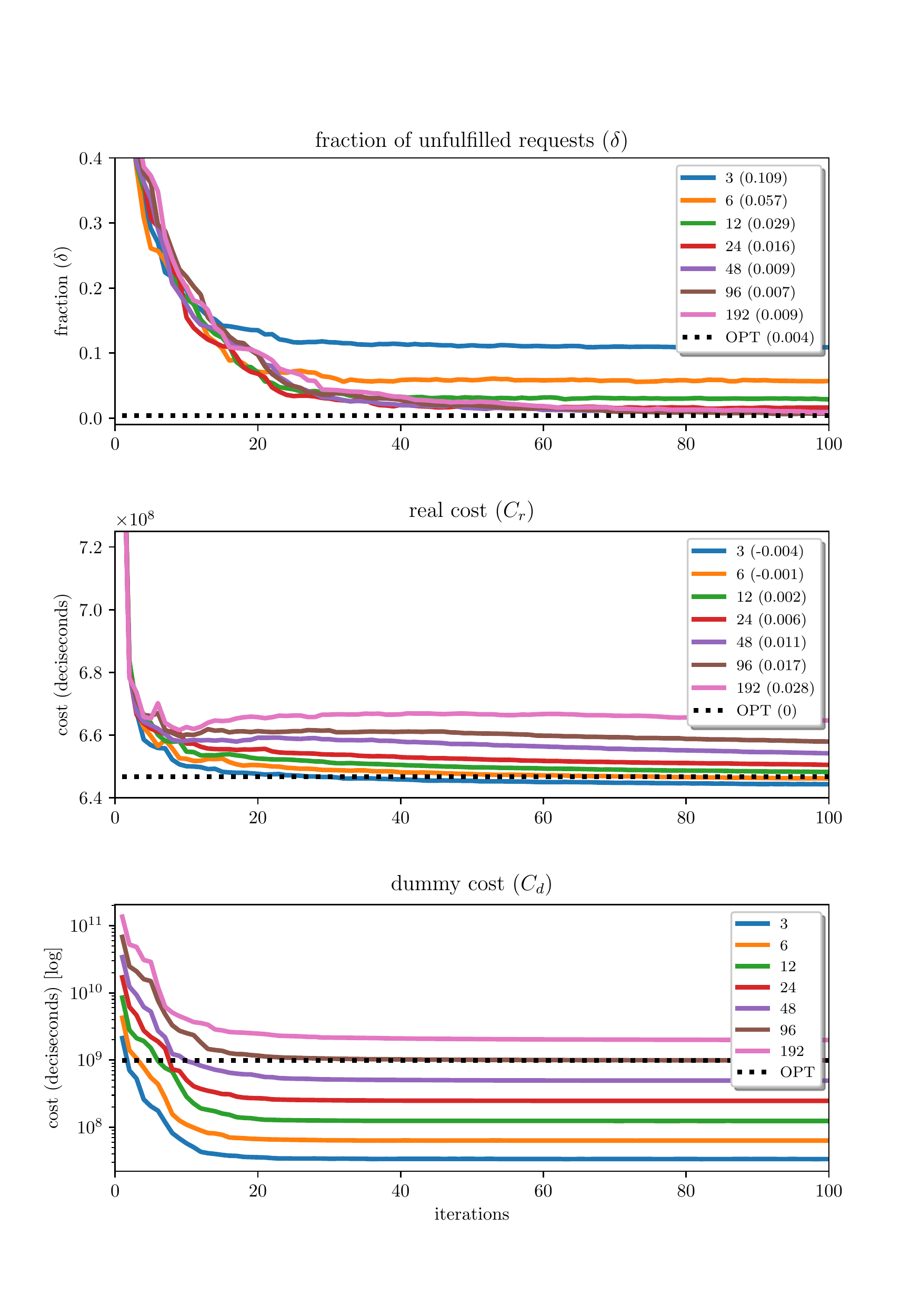}\vspace{-5pt}
    \caption{Validation of theoretical results. For all the plots, OPT represents the corresponding optimal value for $L=96$. [TOP] A plot of $\delta$. The left number in the legend represents $L$, whereas the value in the brackets denotes $\delta$ after iteration $100$. For instance, OPT yields $\delta=0.004$, whereas for the same $L$ after $100$ iterations we have that $\delta=0.007$.
    [CENTER] A plot of $C_r$. The left number in the legend represents $L$, whereas the value in brackets denotes $(C_r-\text{OPT})/\text{OPT}$ after iteration $100$. E.g., $C_r$ for $L=96$ is only $\num{1.7}\%$ longer than OPT. 
    [BOTTOM] A plot for $C_d$. }
	\label{fig:basic}
\vspace{-15pt}
\end{figure}
\vspace{5pt}

\noindent \textbf{Scalability.} In this set of experiments we demonstrate the scalability of the approach by increasing the total number of OD requests. We use the set of $6\times\num{25960}$ OD pairs, which was utilized in the above experiments, as a basis, and then multiply this set by $1,2,3$ and $4$. The last case includes $\num{623040}$ travel requests. In an attempt to make the four settings similar in terms of the total traffic on the road, we pair each setting with exogenous traffic of $\gamma_\text{exo}\in \{0.8,0.6,0.4,0.2\}$, respectively. 

Each scenario was executed for $1000$ iterations. For the first two plots we chose to show results only for the first $100$ iterations, as the change is very minor from this point on. We also omit a plot for the dummy cost as we found it uninformative. The results are presented in Figure~\ref{fig:scale}.

The rate of convergence and the final result for $\delta,C_r$ behave similarly for the four settings. However, we note that we expect to see a more significant difference for a more realistic data set, as a bigger data set would have more diverse OD pairs. Nevertheless, we emphasize that our four settings do not yield similar solutions with respect to flow distribution, as the scale of flow affects the cost. In terms of the running time, we mention that it scales proportionally to the number of OD pairs (e.g., x$1$ is roughly 4 times quicker than x$4$), and the rate of change with respect to the number of iterations is roughly constant. 

\begin{figure}%[h!]
%	\vspace{-.1in}
    \centering
    \includegraphics[width=0.9\columnwidth,trim={0.5cm 1.5cm 1.5cm 2.5cm},clip]{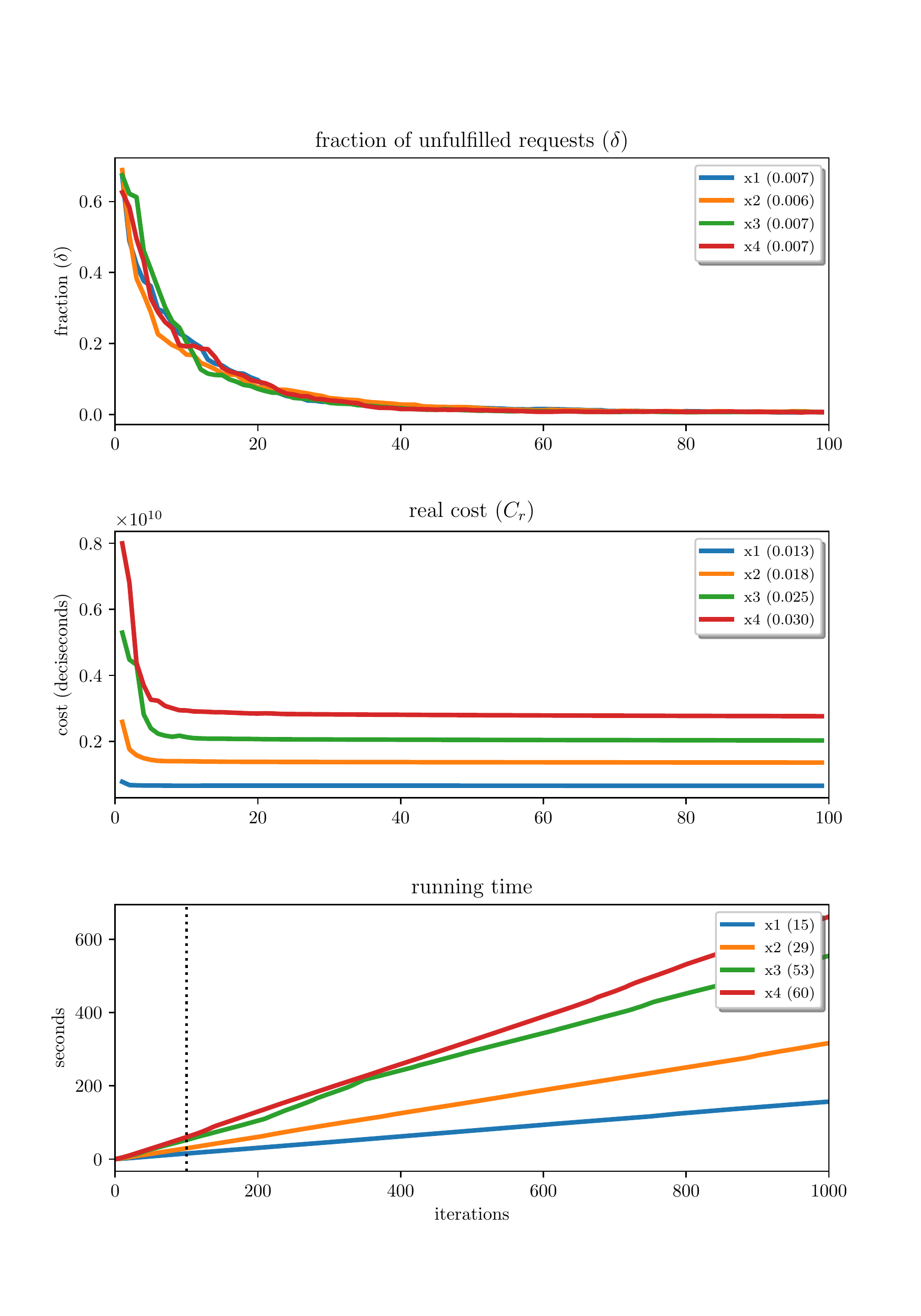}\vspace{-5pt}
    \caption{Plots for scalability experiments.  [TOP] A plot of $\delta$, i.e., fraction of unfulfilled rebalancing requests. The left notation in the legend ($i$x) indicates the used number of copies of the original OD set, whereas the value in brackets denotes $\delta$ after iteration $100$. 
    [CENTER] A plot of $C_r$. The value in brackets denotes the ratio of deviation 
    for $C_r$ between the iterations $100$ and $1000$. 
    [BOTTOM] A plot of total running time. The value in brackets denotes the running time after iteration $100$.}
	\label{fig:scale}
\vspace{-15pt}
\end{figure}\vspace{5pt}

\noindent \textbf{Comparison with previous work.} Here we demonstrate the benefits of computing a solution to \amod using the precise formulation of the cost function \bpr, as opposed to a piecewise-affine approximation of \bpr, or a congestion-unaware solution. The piecewise-affine approximation approach was suggested in~\cite{SalazarTsaoEtAl2019}, where the \bpr function is approximated by two affine functions (see more details in Section~\ref{sec:related}). In terms of computation time, our approach yields the results in around 15 seconds, whereas the reported running times of~\cite{SalazarTsaoEtAl2019} are below 4 minutes for similar hardware. We wish to point out that~\cite{SalazarTsaoEtAl2019} do not minimize travel time for rebalancing vehicles, and also include walking times in their analysis.  The congestion-unaware approach, which was utilized in earlier works (see, e.g., \cite{PavoneSmithEtAl2011}), generates a solution without considering the effect of congestion of the routed vehicles on the overall cost.

We implemented all three types of solution using our framework, by replacing the precise formulation of the \bpr function, where relevant. We wish to clarify that the cost function used on the dummy edges remains \bpr, to guarantee rebalancing (see Section~\ref{sec:theory}), as this does not affect the real solution cost. The computation was done with $L=96$. The same applies to OPT which was computed using the full \bpr.

We ran the three approaches for varying values of \mbox{$\gamma_{\text{exo}}\in [0,2]$} (see plots of the comparison in Figure~\ref{fig:exo}). As was observed in previous studies, congestion-unaware planning underestimates the real travel cost, which results in plans that divert traffic to overly congested routes. The deviation from OPT increases with exogenous traffic. Already for $\gamma_{\text{exo}}=0.8$, the total cost is around $1.3$ times OPT. Approximate \bpr is much more accurate in this respect. However, it either under- or overestimates the real cost, which yields plans where vehicles are rerouted from low-cost routes to more congested routes.  Approximate \bpr yields plans whose deviation from OPT is twice as high for the precise \bpr, when $\gamma_{\text{exo}}\in [0, 1.1]$; it coincides with our solution for $\gamma_{\text{exo}}=1.2$; for larger values of $\gamma_{\text{exo}}\in [0, 1.1]$ the deviation becomes more noticeable. For instance, when  $\gamma_{\text{exo}}=1.3$ it yields a solution of around~$1.1$ times OPT. In contrast, our approach yields an accurate estimation of OPT for the entire range of~$\gamma_{\text{exo}}$. 
\begin{figure}%[h!]
%	\vspace{-.1in}
    \centering
    \includegraphics[width=0.9\columnwidth,trim={0.5cm 16cm 1.4cm 2.5cm},clip]{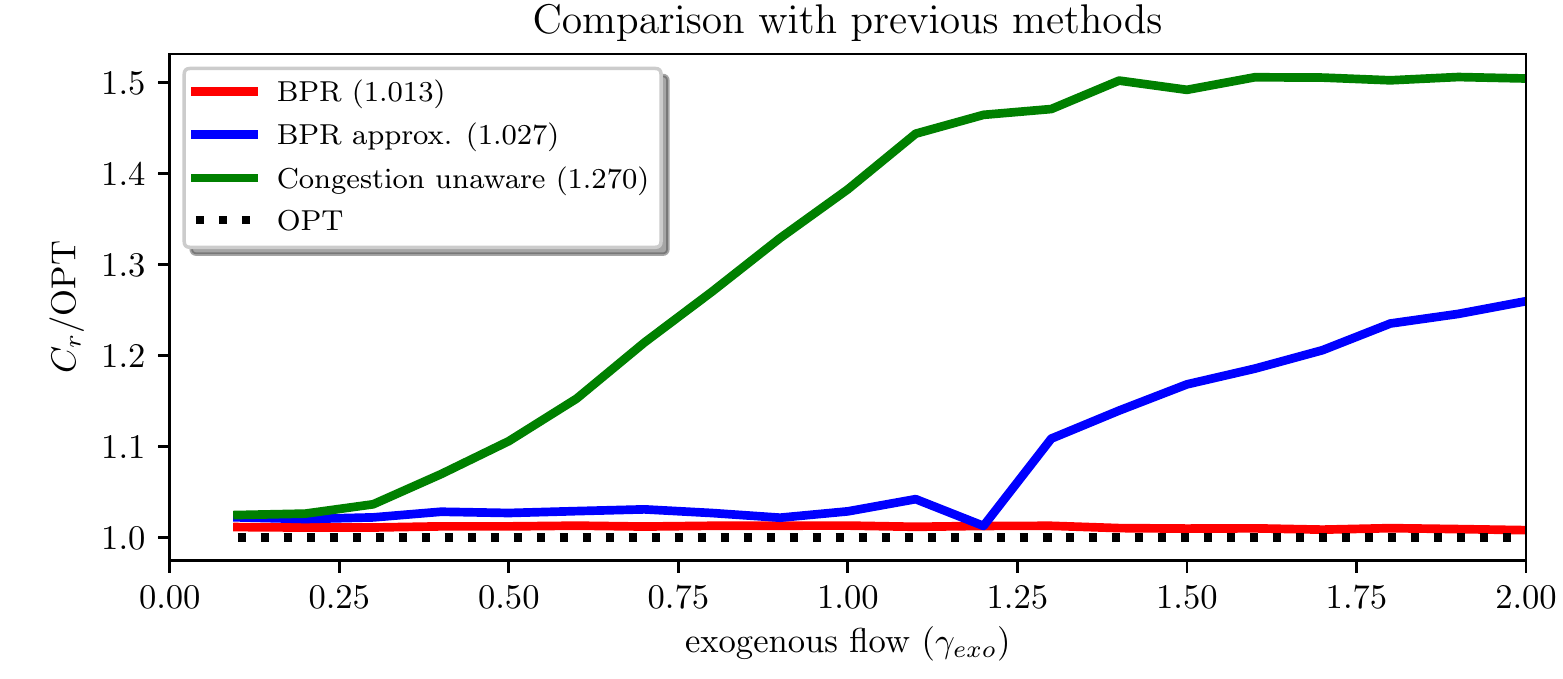}\vspace{-5pt}
    \caption{Plots for comparison experiments. For each of the three types of cost functions we plot the ratio between the obtained cost $C_r$ and the cost for OPT, as a function $\gamma_{\text{exo}}$ ($x$-axis). The corresponding value for $\gamma_{\text{exo}}=0.8$ is given in the brackets.}
	\label{fig:exo}
\vspace{-15pt}
\end{figure}
\section{Discussion and Future Work}\label{sec:future}
This paper presented a computationally-efficient framework to compute the system-optimal operation of a fleet of self-driving cars providing on-demand mobility in a congestion-aware fashion.
To the best of our knowledge, this is the first scheme providing high-quality routing solutions for large-scale \gls{abk:amod} fleets within seconds, thus enabling real-time implementations for operational control through receding horizon optimization (to account for new information that is revealed over time) and randomized sampling (to recover  integer flows with near-optimality guarantees \cite{Rossi2018}).
Our approach consists of reducing our problem to a \tap instance, such that an optimal solution achieved for the latter yields an optimal for \amod. This allows us to leverage modern and highly effective approaches for \tap.
We showcased the benefits of our approach in a real-world case-study. The results showed our method outperforming state-of-the-art approaches by at least one order of magnitude in terms of computational time, whilst improving the system performance by up to 20\%.

This work opens the field for several extensions, and leaves a few open questions.
On the side of theory, our immediate goal is to analyze the convergence rate of our approach. There is an abundance of recent results on the theoretical properties of \frankwolfe (see, e.g.,~\cite{LacosteJaggi15,PenaRodriguez16}), which regained popularity in recent years due to applications in machine learning\footnote{See \url{https://sites.google.com/site/frankwolfegreedytutorial/}.}. However, it is not clear whether these results are applicable to our setting, and what their implications are on the convergence of the real cost $F_E$, rather than $F_{E'}$. We also plan to investigate approaches to obtain a more informative estimation of the constant $L$ (see Theorem~\ref{thm:amod}).
On the implementation side, we mention that the performance can be further improved by using customizable contraction hierarchies~\cite{DibbeltETAL16} for pathfinding, and bundling identical OD pairs. We also point out the fact that the \allornothing routine is embarrassingly parallel, and a significant speedup can be gained from using a multi-core machine.
Finally, on the application side, we aim at exploiting the high computational efficiency of the presented approach by implementing it in real-time using receding horizon schemes.
To this end, it might be necessary to employ a time-expansion of the road-graph and account for stochastic effects, as it was done in~\cite{IglesiasRossiEtAl2016,TsaoIglesiasEtAl2018}.
In addition, it is of interest to extend this framework to capture the interaction with public transit~\cite{SalazarRossiEtAl2018} and the power grid~\cite{RossiIglesiasEtAl2018}, and account for the interaction of self-driving vehicles with the urban infrastructure.

%\todo{Mauro, can you finish the next part? It would be best if you write it.}
%\begin{itemize}
%    \item Outlook motivated by the significantly improved computational performance:
%    \begin{itemize}
%        \item Real-time implementation (MATSim)?
%        \item Extension to the time-varying formulation
%        \item Implement in interaction with the urban infrastructure, e.g., public transit or the power grid
%        \item Mention that in reality a full rebalancing is necessary. However, this is not a problem to achieve using our algorithm, as one can a priori reroute spare rebalancers to fulfill the demand. The impact of this on the actual cost will be negligible due to small number of unfulfilled requests. 
%    \end{itemize}
%\end{itemize}

\ifdoubleblind
\else
\section*{Acknowledgments}
We would like to thank Valentin Buchhold for advising on the {\tt routing-framework}. 
We thank Dr.\ Ilse New and Michal Kleinbort for proofreading this paper and providing us with their comments and advice. Our special thanks are also extended to our labmates Ramón Darío Iglesias, Matthew Tsao, and Jannik Zgraggen for the fruitful discussions. 
The second author would like to express his gratitude to Dr.\ Chris Onder for his support.
This research was supported by the National Science Foundation under CAREER Award CMMI-1454737 and the Toyota Research Institute (TRI). The first author is also supported by the Fulbright Scholars Program. This article solely reflects the opinions and conclusions of its authors and not NSF, TRI, Fulbright, or any other entity.
\fi

\bibliographystyle{plainnat}
\bibliography{tex/bibliography,tex/ASL_papers,main}

\newcommand{\noopsort}[1]{} \newcommand{\printfirst}[2]{#1}
  \newcommand{\singleletter}[1]{#1} \newcommand{\switchargs}[2]{#2#1}
\begin{thebibliography}{53}
\providecommand{\natexlab}[1]{#1}
\providecommand{\url}[1]{\texttt{#1}}
\expandafter\ifx\csname urlstyle\endcsname\relax
  \providecommand{\doi}[1]{doi: #1}\else
  \providecommand{\doi}{doi: \begingroup \urlstyle{rm}\Url}\fi

\bibitem[CIA(2018)]{CIA:2018}
The world factbook.
\newblock Central Intelligence Agency, 2018.
\newblock {Available at }
  \url{https://www.cia.gov/library/publications/the-world-factbook/fields/2212.html}.

\bibitem[Adler et~al.(2015)Adler, De~Berg, Halperin, and Solovey]{AdlerETAL15}
A.~Adler, M.~De~Berg, D.~Halperin, and K.~Solovey.
\newblock Efficient multi-robot motion planning for unlabeled discs in simple
  polygons.
\newblock \emph{{IEEE} Trans. Automation Science and Engineering}, 12\penalty0
  (4):\penalty0 1309--1317, 2015.

\bibitem[Ahuja et~al.(1993)Ahuja, Magnanti, and Orlin]{Ahuja93}
R.~K. Ahuja, T.~L. Magnanti, and J.~B. Orlin.
\newblock \emph{Network Flows: Theory, Algorithms, and Applications}.
\newblock Pearson, 1993.

\bibitem[Akcelik(1978)]{Akcelik78}
R~Akcelik.
\newblock A new look at {D}avidson's travel time function.
\newblock \emph{Traffic Engineering \& Control}, 19\penalty0 (N10), 1978.

\bibitem[Aslam et~al.(2012)Aslam, Lim, and Rus]{AslamLimEtAl2012}
J.~Aslam, S.~Lim, and D.~Rus.
\newblock Congestion-aware traffic routing system using sensor data.
\newblock In \emph{{Proc.\ IEEE Int.\ Conf.\ on Intelligent Transportation
  Systems}}, 2012.

\bibitem[Bast et~al.(2016)Bast, Delling, Goldberg, M{\"u}ller-Hannemann, Pajor,
  Sanders, Wagner, and Werneck]{BastETAL2016}
H.~Bast, D.~Delling, A.~Goldberg, M.~M{\"u}ller-Hannemann, T.~Pajor,
  P.~Sanders, D.~Wagner, and R.~F. Werneck.
\newblock \emph{Route Planning in Transportation Networks}, pages 19--80.
\newblock Springer International Publishing, 2016.

\bibitem[Berger(2018)]{Berger2018}
P.~Berger.
\newblock {MTA} blames {Uber} for decline in {New York City} subway, bus
  ridership.
\newblock The Wall Street Journal, 2018.
\newblock available online.

\bibitem[Buchhold et~al.(2018)Buchhold, Sanders, and Wagner]{BuchETAL18}
V.~Buchhold, P.~Sanders, and D.~Wagner.
\newblock Real-time traffic assignment using fast queries in customizable
  contraction hierarchies.
\newblock In \emph{International Symposium on Experimental Algorithms}, pages
  27:1--27:15, 2018.
\newblock URL \url{https://github.com/vbuchhold/routing-framework}.

\bibitem[{Bureau of Public Roads}(1964)]{BPR1964}
{Bureau of Public Roads}.
\newblock Traffic assignment manual.
\newblock Technical report, {U.S. Dept.\ of Commerce, Urban Planning Division},
  1964.

\bibitem[C{\'{a}}p and Alonso{-}Mora(2018)]{CapAlonso18}
M.~C{\'{a}}p and J.~Alonso{-}Mora.
\newblock Multi-objective analysis of ridesharing in automated
  mobility-on-demand.
\newblock In \emph{Robotics: Science and Systems}, 2018.

\bibitem[C{\'{a}}p et~al.(2015)C{\'{a}}p, P., Kleiner, and
  Seleck{\'{y}}]{CapETAL15}
M.~C{\'{a}}p, Nov{\'{a}}k P., A.~Kleiner, and M.~Seleck{\'{y}}.
\newblock Prioritized planning algorithms for trajectory coordination of
  multiple mobile robots.
\newblock \emph{{IEEE} Trans. Automation Science and Engineering}, 12\penalty0
  (3):\penalty0 835--849, 2015.

\bibitem[Dibbelt et~al.(2016)Dibbelt, Strasser, and Wagner]{DibbeltETAL16}
J.~Dibbelt, B.~Strasser, and D.~Wagner.
\newblock Customizable contraction hierarchies.
\newblock \emph{{ACM} Journal of Experimental Algorithmics}, 21\penalty0
  (1):\penalty0 1.5:1--1.5:49, 2016.
\newblock URL \url{https://github.com/RoutingKit/RoutingKit}.

\bibitem[Frank and Wolfe(1956)]{FrankWolfe56}
M.~Frank and P.~Wolfe.
\newblock An algorithm for quadratic programming.
\newblock \emph{Naval research logistics quarterly}, 3\penalty0 (1-2):\penalty0
  95--110, 1956.

\bibitem[Geisberger et~al.(2012)Geisberger, Sanders, Schultes, and
  Vetter]{GeisbergerETAL12}
R.~Geisberger, P.~Sanders, D.~Schultes, and C.~Vetter.
\newblock Exact routing in large road networks using contraction hierarchies.
\newblock \emph{Transportation Science}, 46\penalty0 (3):\penalty0 388--404,
  2012.

\bibitem[Haklay and Weber(2008)]{HakWeb08}
M.~Haklay and P.~Weber.
\newblock {OpenStreetMap}: User-generated street maps.
\newblock \emph{{IEEE} Pervasive Computing}, 7\penalty0 (4):\penalty0 12--18,
  2008.

\bibitem[H\"orl et~al.(2018)H\"orl, Ruch, Becker, Frazzoli, and
  Axhausen]{HoerlRuchEtAl2018}
S.~H\"orl, C.~Ruch, F.~Becker, E.~Frazzoli, and K.~W. Axhausen.
\newblock Fleet control algorithms for automated mobility: A simulation
  assessment for {Z}urich.
\newblock In \emph{{Annual Meeting of the Transportation Research Board}},
  2018.

\bibitem[Hu(2017)]{Hu2017}
W.~Hu.
\newblock Your {Uber} car creates congestion. should you pay a fee to ride?
\newblock The New York Times, 2017.
\newblock available online.

\bibitem[Iglesias et~al.(2016)Iglesias, Rossi, Zhang, and
  Pavone]{IglesiasRossiEtAl2016}
R.~Iglesias, F.~Rossi, R.~Zhang, and M.~Pavone.
\newblock A {BCMP} network approach to modeling and controlling {Autonomous}
  {Mobility-on-Demand} systems.
\newblock In \emph{{Workshop on Algorithmic Foundations of Robotics}}, 2016.

\bibitem[Iglesias et~al.(2018)Iglesias, Rossi, Wang, Hallac, Leskovec, and
  Pavone]{IglesiasRossiEtAl2018}
R.~Iglesias, F.~Rossi, K.~Wang, D.~Hallac, J.~Leskovec, and M.~Pavone.
\newblock Data-driven model predictive control of autonomous mobility-on-demand
  systems.
\newblock In \emph{{Proc.\ IEEE Conf.\ on Robotics and Automation}}, 2018.

\bibitem[Lacoste{-}Julien and Jaggi(2015)]{LacosteJaggi15}
S.~Lacoste{-}Julien and M.~Jaggi.
\newblock On the global linear convergence of {F}rank-{W}olfe optimization
  variants.
\newblock In \emph{Neural Information Processing Systems}, pages 496--504,
  2015.

\bibitem[Levin et~al.(2017)Levin, Kockelman, Boyles, and
  Li]{LevinKockelmanEtAl2017}
M.~W. Levin, K.~M. Kockelman, S.~D. Boyles, and T.~Li.
\newblock A general framework for modeling shared autonomous vehicles with
  dynamic network-loading and dynamic ride-sharing application.
\newblock \emph{Computers, Environment and Urban Systems}, 64:\penalty0 373 --
  383, 2017.

\bibitem[Levy et~al.(2010)Levy, Buonocore, and
  Von~Stackelberg]{LevyBuonocoreEtAl2010}
J.~I. Levy, J.~J. Buonocore, and K.~Von~Stackelberg.
\newblock Evaluation of the public health impacts of traffic congestion: a
  health risk assessment.
\newblock \emph{Environmental Health}, 9\penalty0 (1):\penalty0 65, 2010.

\bibitem[Maciejewski et~al.(2017)Maciejewski, Bischoff, H\"orl, and
  Nagel]{MaciejewskiBischoffEtAl2017}
M.~Maciejewski, J.~Bischoff, S.~H\"orl, and K.~Nagel.
\newblock Towards a testbed for dynamic vehicle routing algorithms.
\newblock In \emph{{Int.\ Conf.\ on Practical Applications of Agents and
  Multi-Agent Systems - Workshop on the application of agents to passenger
  transport (PAAMS-TAAPS)}}, 2017.

\bibitem[Manley et~al.(2014)Manley, Cheng, Penn, and
  Emmonds]{ManleyChengEtAl2014}
E.~Manley, T.~Cheng, A.~Penn, and A.~Emmonds.
\newblock A framework for simulating large-scale complex urban traffic dynamics
  through hybrid agent-based modelling.
\newblock \emph{{Computers, Environment and Urban Systems}}, 44:\penalty0
  {27--36}, 2014.

\bibitem[Molla(2018)]{Molla2018}
R.~Molla.
\newblock {Americans seem to like ride-sharing services like Uber and Lyft. But
  it's hard to say exactly how many use them.}
\newblock Recode, 2018.
\newblock {Available at
  }\url{https://www.recode.net/2018/6/24/17493338/ride-sharing-services-uber-lyft-how-many-people-use}.

\bibitem[Neuburger(1971)]{Neuburger1971}
H.~Neuburger.
\newblock The economics of heavily congested roads.
\newblock \emph{{Transportation Research}}, 5\penalty0 (4):\penalty0 283--293,
  1971.

\bibitem[Patriksson(2015)]{Patriksson15}
Michael Patriksson.
\newblock \emph{The Traffic Assignment Problem: Models and Methods}.
\newblock Dover Publications, 2015.
\newblock ISBN 0486787907.

\bibitem[Pavone et~al.(2011)Pavone, Smith, Frazzoli, and
  Rus]{PavoneSmithEtAl2011}
M.~Pavone, S.~L. Smith, E.~Frazzoli, and D.~Rus.
\newblock Load balancing for {Mobility-on-Demand} systems.
\newblock In \emph{{Robotics: Science and Systems}}, 2011.

\bibitem[Pavone et~al.(2012)Pavone, Smith, Frazzoli, and
  Rus]{PavoneSmithEtAl2012}
M.~Pavone, S.~L. Smith, E.~Frazzoli, and D.~Rus.
\newblock Robotic load balancing for {Mobility-on-Demand} systems.
\newblock \emph{{Int.\ Journal of Robotics Research}}, 31\penalty0
  (7):\penalty0 839--854, 2012.

\bibitem[Pe{\~{n}}a et~al.(2016)Pe{\~{n}}a, Rodr{\'{\i}}guez, and
  Soheili]{PenaRodriguez16}
J.~Pe{\~{n}}a, D.~Rodr{\'{\i}}guez, and N.~Soheili.
\newblock On the von {N}eumann and {F}rank-{W}olfe algorithms with away steps.
\newblock \emph{{SIAM} Journal on Optimization}, 26\penalty0 (1):\penalty0
  499--512, 2016.

\bibitem[Rivas et~al.(2016)Rivas, Inmaculada, {S\'anchez-Cambronero}, Barba,
  and {Ruiz-Ripoll}]{RivasInmaculadaEtAl2016}
A.~Rivas, G.~Inmaculada, S.~{S\'anchez-Cambronero}, R.~M. Barba, and
  L.~{Ruiz-Ripoll}.
\newblock A continuous dynamic traffic assignment model from plate scanning
  technique.
\newblock \emph{{Transport Research Procedia}}, 18:\penalty0 {332--340}, 2016.

\bibitem[Rossi(2018)]{Rossi2018}
F.~Rossi.
\newblock \emph{On the Interaction between {Autonomous Mobility-on-Demand}
  Systems and the Built Environment: Models and Large Scale Coordination
  Algorithms}.
\newblock PhD thesis, {Stanford University, Dept.\ of Aeronautics and
  Astronautics}, 2018.

\bibitem[Rossi et~al.(2018{\natexlab{a}})Rossi, Iglesias, Alizadeh, and
  Pavone]{RossiIglesiasEtAl2018}
F.~Rossi, R.~Iglesias, M.~Alizadeh, and M.~Pavone.
\newblock On the interaction between {Autonomous Mobility-on-Demand} systems
  and the power network: Models and coordination algorithms.
\newblock In \emph{{Robotics: Science and Systems}}, 2018{\natexlab{a}}.
\newblock {Extended version available at
  }\url{https://arxiv.org/abs/1709.04906}.

\bibitem[Rossi et~al.(2018{\natexlab{b}})Rossi, Zhang, Hindy, and
  Pavone]{RossiZhangEtAl2017}
F.~Rossi, R.~Zhang, Y.~Hindy, and M.~Pavone.
\newblock Routing autonomous vehicles in congested transportation networks:
  Structural properties and coordination algorithms.
\newblock \emph{{Autonomous Robots}}, 42\penalty0 (7):\penalty0 1427--1442,
  2018{\natexlab{b}}.

\bibitem[Salazar et~al.(2018)Salazar, Rossi, Schiffer, Onder, and
  Pavone]{SalazarRossiEtAl2018}
M.~Salazar, F.~Rossi, M.~Schiffer, C.~H. Onder, and M.~Pavone.
\newblock On the interaction between autonomous mobility-on-demand and the
  public transportation systems.
\newblock In \emph{{Proc.\ IEEE Int.\ Conf.\ on Intelligent Transportation
  Systems}}, 2018.
\newblock In Press. {Extended Version, Available} at
  \url{https://arxiv.org/abs/1804.11278}.

\bibitem[Salazar et~al.(2019)Salazar, Tsao, Aguiar, Schiffer, and
  Pavone]{SalazarTsaoEtAl2019}
M.~Salazar, M.~Tsao, I.~Aguiar, M.~Schiffer, and M.~Pavone.
\newblock A congestion-aware routing scheme for autonomous mobility-on-demand
  systems.
\newblock In \emph{{European Control Conference}}, 2019.
\newblock Submitted.

\bibitem[Schrank et~al.(2007)Schrank, Lomax, and Turner]{Schrank2007}
D.~Schrank, T.~Lomax, and S.~Turner.
\newblock The 2007 urban mobility report.
\newblock Technical report, Texas Transportation Institute, 2007.

\bibitem[Sharon et~al.(2015)Sharon, Stern, Felner, and
  Sturtevant]{SharonETAL15}
G.~Sharon, R.~Stern, A.~Felner, and N.~R. Sturtevant.
\newblock Conflict-based search for optimal multi-agent pathfinding.
\newblock \emph{Artificial Intelligence}, 219:\penalty0 40--66, 2015.

\bibitem[Sheffi(1985)]{Sheffi85}
Y.~Sheffi.
\newblock \emph{Urban Transportation Networks: Equilibrium Analysis with
  Mathematical Programming Methods}.
\newblock Prentice Hall, 1985.

\bibitem[Shome et~al.(2019)Shome, Solovey, Dobson, Halperin, and
  Bekris]{ShomeETAL2019}
R.~Shome, K.~Solovey, A.~Dobson, D.~Halperin, and K.~E. Bekris.
\newblock {dRRT*}: Scalable and informed asymptotically-optimal multi-robot
  motion planning.
\newblock \emph{Autonomous Robots}, 2019.

\bibitem[Siddiqui(2018)]{Siddiqui2018}
F.~Siddiqui.
\newblock Failing transit ridership poses an {`emergency'} for cities, experts
  fear.
\newblock The Washington Post, 2018.
\newblock {available online}.

\bibitem[Solovey et~al.(2015)Solovey, Yu, Zamir, and Halperin]{SoloveyETAL15}
K.~Solovey, J.~Yu, O.~Zamir, and D.~Halperin.
\newblock Motion planning for unlabeled discs with optimality guarantees.
\newblock In \emph{Robotics: {S}cience and {S}ystems}, 2015.

\bibitem[Solovey and Halperin(2016)]{SolHal16}
Kiril Solovey and Dan Halperin.
\newblock On the hardness of unlabeled multi-robot motion planning.
\newblock \emph{I. J. Robotics Res.}, 35\penalty0 (14):\penalty0 1750--1759,
  2016.

\bibitem[Spieser et~al.(2014)Spieser, Treleaven, Zhang, Frazzoli, Morton, and
  Pavone]{SpieserTreleavenEtAl2014}
K.~Spieser, K.~Treleaven, R.~Zhang, E.~Frazzoli, D.~Morton, and M.~Pavone.
\newblock Toward a systematic approach to the design and evaluation of
  {Autonomous} {Mobility-on-Demand} systems: A case study in {Singapore}.
\newblock In \emph{Road Vehicle Automation}. {Springer}, 2014.

\bibitem[Tsao et~al.(2018)Tsao, Iglesias, and Pavone]{TsaoIglesiasEtAl2018}
M.~Tsao, R.~Iglesias, and M.~Pavone.
\newblock Stochastic model predictive control for autonomous mobility on
  demand.
\newblock In \emph{{Proc.\ IEEE Int.\ Conf.\ on Intelligent Transportation
  Systems}}, 2018.
\newblock In Press. {Extended Version, Available} at
  \url{https://arxiv.org/pdf/1804.11074}.

\bibitem[Tsao et~al.(2019)Tsao, Milojevic, Ruch, Salazar, Frazzoli, and
  Pavone]{TsaoMilojevicEtAl2019}
M.~Tsao, D.~Milojevic, C.~Ruch, M.~Salazar, E.~Frazzoli, and M.~Pavone.
\newblock Model predictive control of ride-sharing autonomous mobility on
  demand systems.
\newblock In \emph{{Proc.\ IEEE Conf.\ on Robotics and Automation}}, 2019.

\bibitem[Turpin et~al.(2014{\natexlab{a}})Turpin, Michael, and
  Kumar]{TurpinETAL14b}
M.~Turpin, N.~Michael, and V.~Kumar.
\newblock Capt: Concurrent assignment and planning of trajectories for multiple
  robots.
\newblock \emph{I. J. Robotics Res.}, 33\penalty0 (1):\penalty0 98--112,
  2014{\natexlab{a}}.

\bibitem[Turpin et~al.(2014{\natexlab{b}})Turpin, Mohta, Michael, and
  Kumar]{TurpinETAL14a}
M.~Turpin, K.~Mohta, N.~Michael, and V.~Kumar.
\newblock Goal assignment and trajectory planning for large teams of
  interchangeable robots.
\newblock \emph{Auton. Robots}, 37\penalty0 (4):\penalty0 401--415,
  2014{\natexlab{b}}.

\bibitem[Tuttle and Cowles(2014)]{TuttleCowles2014}
B.~Tuttle and T.~Cowles.
\newblock Traffic jams cost {A}mericans \$124 billion in 2013.
\newblock \emph{Time - Money}, 2014.

\bibitem[Yu(2018)]{Yu18}
J.~Yu.
\newblock Constant factor time optimal multi-robot routing on high-dimensional
  grids.
\newblock In \emph{Robotics: Science and Systems}, 2018.

\bibitem[Yu and LaValle(2012)]{YuLaValle12}
J.~Yu and S.~M. LaValle.
\newblock Distance optimal formation control on graphs with a tight convergence
  time guarantee.
\newblock In \emph{{IEEE} Conference on Decision and Control}, pages
  4023--4028, 2012.

\bibitem[Yu and LaValle(2016)]{YuLaValle16}
J.~Yu and S.~M. LaValle.
\newblock Optimal multirobot path planning on graphs: Complete algorithms and
  effective heuristics.
\newblock \emph{{IEEE} Trans. Robotics}, 32\penalty0 (5):\penalty0 1163--1177,
  2016.

\bibitem[Zhang and Pavone(2016)]{ZhangPavone2016}
R.~Zhang and M.~Pavone.
\newblock Control of robotic {Mobility-on-Demand} systems: A
  queueing-theoretical perspective.
\newblock \emph{{Int.\ Journal of Robotics Research}}, 35\penalty0
  (1--3):\penalty0 186--203, 2016.

\end{thebibliography}

\end{document}